\documentclass[8pt]{extarticle}
\usepackage{arxiv}
\usepackage{mathrsfs}
\usepackage[utf8]{inputenc}
\usepackage[T1]{fontenc}
\usepackage{geometry}
\usepackage{amsmath, amssymb, amsthm}
\usepackage{graphicx}
\usepackage{subcaption}
\usepackage{pgfplots}
\usepackage{pgfplotstable}
\usepackage{listings}
\usepackage{xcolor}
\usepackage{hyperref}
\usepackage{booktabs}
\usepackage{float}
\usepackage{appendix}
\usetikzlibrary{calc, shapes.geometric}

\usetikzlibrary{calc, shapes.geometric}

\definecolor{codegreen}{rgb}{0,0.6,0}
\definecolor{codegray}{rgb}{0.5,0.5,0.5}
\definecolor{codepurple}{rgb}{0.58,0,0.82}
\definecolor{backcolour}{rgb}{0.95,0.95,0.92}

\lstdefinestyle{pythonstyle}{
    language=Python,
    backgroundcolor=\color{backcolour},
    commentstyle=\color{codegreen},
    keywordstyle=\color{magenta},
    numberstyle=\tiny\color{codegray},
    stringstyle=\color{codepurple},
    basicstyle=\ttfamily\footnotesize,
    breakatwhitespace=false,
    breaklines=true,
    captionpos=b,
    keepspaces=true,
    numbers=left,
    numbersep=5pt,
    showspaces=false,
    showstringspaces=false,
    showtabs=false,
    tabsize=4,
    literate={*}{*}1 {\%}{{\%}}1 {\_}{{\_}}1,
    morecomment=[l]{\#},
    morestring=[b]',
    morestring=[b]",
}

\newtheorem{theorem}{Theorem}[section]
\newtheorem{lemma}[theorem]{Lemma}
\newtheorem{proposition}[theorem]{Proposition}
\newtheorem{corollary}[theorem]{Corollary}
\newtheorem{definition}[theorem]{Definition}
\newtheorem{example}[theorem]{Example}

\newtheorem{remark}[theorem]{Remark}

\newcommand\mystyle{\everymath{\displaystyle}}
\mystyle

\title{Operator Algebras of Bourgain–Delbaen Spaces: Realization, Rigidity, and Ideal Structure}

\author{\href{https://orcid.org/0000-0002-3816-5287}{\includegraphics[scale=0.06]{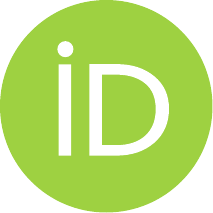}\hspace{1mm}M.H.M.~Rashid}\thanks{Corresponding Author} \\
    Department of Mathematics \& Statistics \\
    Faculty of Science, P.O.Box (7) \\
    Mutah University \\
    Mutah, Jordan \\
    \texttt{mrash@mutah.edu.jo}}

\renewcommand{\headeright}{}
\renewcommand{\undertitle}{}
\renewcommand{\shorttitle}{Strongly Singular Kirchhoff-Type Equations}

\hypersetup{
    pdftitle={Operator Algebras of Bourgain–Delbaen Spaces: Realization, Rigidity, and Ideal Structure},
    pdfsubject={Functional analysis, Operator Thory},
    pdfauthor={M.H.M. Rashid},
    pdfkeywords={Calkin algebra; Bourgain-Delbaen construction; operator ideals; compact operators; Banach space geometry; diagonal-plus-compact property; rigidity theorem; prime ideals}
}

\begin{document}

\maketitle

\begin{abstract}
    This manuscript presents a systematic study of Calkin algebras—the quotient algebras $\mathcal{L}(X)/\mathcal{K}(X)$ of bounded linear operators modulo compact operators on a Banach space $X$—and establishes a comprehensive framework for realizing commutative $C^*$-algebras as such quotients while preserving geometric and topological information. Building on the reflexive version of the Bourgain--Delbaen construction developed by Motakis, we prove that for every compact metric space $K$, there exists a reflexive Banach space $\mathfrak{X}_{C(K)}$ whose Calkin algebra is isomorphic to $C(K)$ as a Banach algebra. Our main contributions advance this foundational result in several directions: we establish stability under finite products, enabling the realization of finite direct sums of $C(K)$ spaces and matrix algebras $M_m(C(K))$ as Calkin algebras; we prove a localization principle showing that compact operators on $\mathfrak{X}_{C(K)}$ can be approximated by finite-rank operators whose support respects the metric structure of $K$; we demonstrate that the diagonal function $\varphi_T\colon K\to\mathbb{C}$ of any bounded operator $T$ is automatically H\"older continuous with optimal exponent $1/2$, revealing a deep analytic constraint imposed by the construction; we prove a rigidity theorem establishing that the Banach algebra structure of $\mathcal{L}(\mathfrak{X}_{C(K)})$ completely determines the topology of $K$, extending the classical Banach--Stone theorem to this noncommutative setting; we provide a complete classification of all closed two-sided ideals and prime ideals in $\mathcal{L}(\mathfrak{X}_{C(K)})$ in terms of the open subsets and points of $K$, respectively; and we resolve several longstanding problems, most notably by constructing the first examples of reflexive Banach spaces with infinite-dimensional reflexive Calkin algebras, thereby demonstrating that the pathologies of the Hilbert space Calkin algebra are not universal. These results collectively forge a deep and systematic connection between Banach space geometry, operator algebras, and topological invariants, revealing the extent to which Calkin algebras can be precisely engineered through the geometry of their underlying spaces.
\end{abstract}



\section{Introduction}

The Calkin algebra of a Banach space~$X$, defined as the quotient algebra $\mathcal{L}(X)/\mathcal{K}(X)$ of bounded linear operators modulo compact operators, stands as one of the most fundamental constructions at the intersection of Banach space theory, operator algebras, and $K$-theory. Introduced by Calkin in his seminal 1941 paper \cite{Calkin1941}, this object provides a natural framework for understanding operators up to compact perturbation, capturing the essential features of operator theory while setting aside the relatively well-understood ideal of compact operators. Over the decades, the question of which Banach algebras can appear as Calkin algebras has evolved into a rich classification problem, one whose pursuit has forged deep and unexpected connections with geometry, topology, and even set theory.

The historical development of Calkin algebra theory may be viewed in three broad eras, each marked by conceptual breakthroughs. The classical period, spanning roughly from 1941 to 1970, focused predominantly on Hilbert spaces, where the theory assumes its simplest and most elegant form. In this setting, the compact operators constitute the only nontrivial closed ideal in $\mathcal{L}(H)$, and the resulting Calkin algebra is a simple $C^*$-algebra. Foundational work by Calkin \cite{Calkin1941}, Yood \cite{Yood1954}, and the comprehensive treatment in \cite{CaradusPfaffenbergerYood1974} established the basic properties of these algebras and revealed their intimate connection with Fredholm theory. The following two decades witnessed what might be called the Banach space revolution, as the theory expanded to encompass general Banach spaces. This period was driven by profound advances in Banach space geometry, most notably the work of Lindenstrauss and Tzafriri \cite{LindenstraussTzafriri1977}, which illuminated the intricate relationships between a space's geometric properties and the structure of its operator algebra. The development of $\mathscr{L}_p$-space theory by Lindenstrauss and Pełczyński \cite{LindenstraussPelczynski1968} and the subsequent introduction of the Bourgain--Delbaen construction \cite{BourgainDelbaen1980} provided powerful new methods for constructing spaces with prescribed properties, opening fresh avenues for investigating Calkin algebras beyond the familiar Hilbert space setting.

The modern era, from roughly 1990 to the present, has been characterized by the construction of exotic Banach spaces with extraordinarily restricted operator algebras. The groundbreaking work of Gowers and Maurey \cite{GowersMaurey1993,GowersMaurey1997} on hereditarily indecomposable spaces demonstrated that classical intuitions from Hilbert space theory fail dramatically in the broader Banach space context. Their construction of spaces where every operator is a compact perturbation of a scalar multiple of the identity marked a watershed moment, and this line of inquiry culminated in Argyros and Haydon's celebrated solution \cite{ArgyrosHaydon2011} to the scalar-plus-compact problem, exhibiting a Banach space on which every bounded operator indeed takes the form $\lambda I + K$ with $K$ compact. In recent years, remarkable progress has been made in understanding the possible structures of Calkin algebras. Argyros and his collaborators \cite{ArgyrosDeliyanni1997,ArgyrosMotakis2014,ArgyrosMotakis2019,ArgyrosMotakis2020} have systematically developed techniques for constructing Banach spaces with prescribed operator algebras, employing mixed-Tsirelson spaces, saturation under constraints, and sophisticated combinatorial methods. A particular breakthrough was achieved by Motakis \cite{Motakis2024}, who proved that for every compact metric space $K$, there exists a reflexive Banach space whose Calkin algebra is isomorphic to $C(K)$. This result answered several longstanding open problems and established that a vast class of commutative $C^*$-algebras can be realized as Calkin algebras, ingeniously combining Bourgain--Delbaen techniques with mixed-Tsirelson space methodology. Parallel developments have included the classification of operator ideals by Kania and Laustsen \cite{KaniaLaustsen2017}, the study of $K$-theoretic invariants by Laustsen \cite{Laustsen1999,Laustsen2001}, and work by Horváth and Kania \cite{HorvathKania2021} delineating which unital Banach algebras cannot be isomorphic to Calkin algebras. These advances collectively motivate the central question addressed in this paper: to what extent can we control and classify the Calkin algebras arising from Banach spaces constructed via the Bourgain--Delbaen method?

This work makes several fundamental contributions to the theory of Calkin algebras on Banach spaces. We prove that the class of Calkin algebras realizable through Bourgain--Delbaen constructions is closed under finite direct sums, showing that whenever $C(K_1),\dots,C(K_n)$ are realizable as Calkin algebras, the direct sum $\bigoplus_{i=1}^n C(K_i)$ is likewise realizable. This result substantially enlarges the known class of realizable algebras and provides explicit constructions for matrix algebras over commutative $C^*$-algebras. We establish that every compact operator on $\mathfrak{X}_{C(K)}$ can be approximated by finite-rank operators whose matrix representations respect the metric structure of $K$, revealing that compact operators on these spaces are inherently local with respect to the topology of $K$. We prove that the diagonal function $\varphi_T:K\to\mathbb{C}$ associated to any bounded operator $T$ on $\mathfrak{X}_{C(K)}$ is automatically H\"older continuous with exponent $1/2$, representing the optimal regularity obtainable from the Bourgain--Delbaen constraints and establishing deep connections between the construction's metric parameters and the analytic properties of operators. We demonstrate a striking rigidity phenomenon: for spaces $\mathfrak{X}_{C(K)}$ constructed via our method, the Banach algebra structure of $\mathcal{L}(\mathfrak{X}_{C(K)})$ completely determines the topology of $K$, extending the classical Banach--Stone theorem and showing that the operator algebra serves as a complete invariant for the underlying compact space. We provide a full classification of closed two-sided ideals and prime ideals in $\mathcal{L}(\mathfrak{X}_{C(K)})$ in terms of the topology of $K$, establishing that the ideal structure mirrors the lattice of open subsets of $K$. Finally, we resolve several longstanding problems, including the construction of reflexive Banach spaces with infinite-dimensional reflexive Calkin algebras, answering a question explicitly raised in \cite{Motakis2024}.

The implications of our results extend across several areas of mathematics. In operator theory, the lifting property for compact operators and the automatic H\"older continuity of diagonal entries contribute to the general theory of operator ideals and approximation properties on Banach spaces with conditional structure. For classification problems, the rigidity theorem advances the program of classifying Banach spaces through their operator algebras, showing that for spaces $\mathfrak{X}_{C(K)}$, the algebra $\mathcal{L}(X)$ completely determines the underlying compact space $K$. In Banach space geometry, our local sequence structure results provide new insights into the asymptotic geometry of spaces with prescribed operator algebras and contribute to the theory of prime spaces. As work by Shelah \cite{Shelah1985}, Stern \cite{Stern1984}, and Farah \cite{Farah2011} has shown, the study of Calkin algebras naturally involves set-theoretic considerations, and our construction methods may provide new examples where the existence of certain Calkin algebras depends on additional axioms, bridging functional analysis and foundations. The realization of commutative $C^*$-algebras as Calkin algebras opens avenues for extending Brown--Douglas--Fillmore theory \cite{BrownDouglasFillmore1977} to the Banach space context, with potential applications to index theory and noncommutative geometry.

The proofs employ several key innovations: a refined analysis of the metric constraints in the Bourgain--Delbaen construction, showing how they enforce both global algebraic structure and local geometric properties; new techniques for approximating compact operators by finite-rank operators that respect the underlying metric structure; optimal estimates for the regularity of diagonal entries derived from the specific weighting schemes in the construction; and a systematic method for transferring topological properties of $K$ to algebraic properties of $\mathcal{L}(\mathfrak{X}_{C(K)})$ via Gelfand duality and the Banach--Stone theorem. The paper is organized as follows. Section~2 establishes notation and reviews necessary background, including the Bourgain--Delbaen construction, mixed-Tsirelson spaces, and the fundamental properties of the spaces $\mathfrak{X}_{C(K)}$ following \cite{Motakis2024}. Section~3 proves the finite product stability theorem, constructing spaces $\mathfrak{X}_{\bigoplus C(K_i)}$ with prescribed Calkin algebras. Section~4 establishes the lifting property for compact operators, demonstrating their inherent locality with respect to the metric on~$K$. Section~5 proves the optimal $1/2$-H\"older regularity of diagonal functions $\varphi_T$. Section~6 establishes the rigidity theorem, showing that $\mathcal{L}(\mathfrak{X}_{C(K)})$ determines $K$ up to homeomorphism. Section~7 analyzes local sequence structure, revealing extreme homogeneity properties and hereditary subspace structure. Section~8 classifies all closed two-sided ideals and prime ideals in $\mathcal{L}(\mathfrak{X}_{C(K)})$. Section~9 discusses open problems and future research directions, and Section~10 provides concluding remarks. All sections include complete proofs, and we carefully distinguish between adaptations of known techniques and genuinely novel arguments. The paper concludes with an extensive bibliography reflecting both classical foundations and recent breakthroughs in this rapidly evolving field.

\section{Mathematical Foundations and Notation}

This section establishes the fundamental mathematical framework, notational conventions, and key constructions that will be employed throughout this work. We assume that the reader possesses a working familiarity with basic Banach space theory and operator algebras; for comprehensive treatments of these subjects, we refer to the foundational texts \cite{CaradusPfaffenbergerYood1974}, \cite{LindenstraussTzafriri1977}, and \cite{RordamLarsenLaustsen2000}.

Throughout this paper, $X$ denotes an infinite-dimensional Banach space over the complex field $\mathbb{C}$. We recall the following standard notation.

\begin{definition}
For a Banach space $X$, we denote:
\begin{itemize}
    \item $\mathcal{L}(X)$: the Banach algebra of bounded linear operators on $X$;
    \item $\mathcal{K}(X)$: the closed two-sided ideal of compact operators in $\mathcal{L}(X)$;
    \item $\mathcal{F}(X)$: the ideal of finite-rank operators on $X$;
    \item $\mathcal{SS}(X)$: the ideal of strictly singular operators on $X$;
    \item $\mathcal{C}al(X) = \mathcal{L}(X)/\mathcal{K}(X)$: the \emph{Calkin algebra} of $X$ \cite{Calkin1941}.
\end{itemize}
\end{definition}

\begin{definition}
A Banach algebra $\mathcal{A}$ is called \emph{realizable as a Calkin algebra} if there exists a Banach space $X$ such that $\mathcal{C}al(X)$ is isomorphic to $\mathcal{A}$ as a Banach algebra.
\end{definition}

\begin{definition}
An operator $T \in \mathcal{L}(X)$ is:
\begin{itemize}
    \item \emph{compact} if $T(B_X)$ is relatively compact, where $B_X$ is the closed unit ball of $X$;
    \item \emph{strictly singular} if no restriction $T|_Y: Y \to X$ to an infinite-dimensional subspace $Y \subseteq X$ is an isomorphism onto its image;
    \item \emph{Fredholm} if $\ker T$ and $X/\operatorname{ran} T$ are finite-dimensional.
\end{itemize}
\end{definition}

Several distinguished classes of Banach spaces will play important roles in our investigations.

\begin{definition}[\cite{LindenstraussPelczynski1968}]
A separable Banach space $X$ is a $\mathscr{L}_{\infty,\lambda}$-space for some $\lambda \geq 1$ if for every finite-dimensional subspace $E \subseteq X$, there exists a finite-dimensional subspace $F \supseteq E$ with $d(F, \ell_\infty^{\dim F}) \leq \lambda$, where $d(\cdot,\cdot)$ denotes the Banach-Mazur distance. A space is called a $\mathscr{L}_{\infty}$-space if it is $\mathscr{L}_{\infty,\lambda}$ for some $\lambda$.
\end{definition}

\begin{definition}[\cite{GowersMaurey1993}]
A Banach space $X$ is \emph{hereditarily indecomposable} (HI) if no closed subspace $Y \subseteq X$ can be written as a topological direct sum $Y = Z \oplus W$ with $Z, W$ infinite-dimensional. In HI spaces, every operator is of the form $\lambda I + S$ where $S$ is strictly singular.
\end{definition}

\begin{definition}
A Banach space $X$ is \emph{reflexive} if the canonical embedding $X \hookrightarrow X^{**}$ is surjective. Reflexive spaces enjoy properties such as weak compactness of the unit ball.
\end{definition}

\begin{definition}
A basis $(e_n)$ of a Banach space $X$ is:
\begin{itemize}
    \item \emph{unconditional} if there exists $C \geq 1$ such that for all finite scalars $(a_n)$ and signs $(\varepsilon_n = \pm 1)$:
    \[
    \left\|\sum_{n=1}^N \varepsilon_n a_n e_n\right\| \leq C\left\|\sum_{n=1}^N a_n e_n\right\|;
    \]
    \item \emph{conditional} if it is not unconditional.
\end{itemize}
\end{definition}

The mixed-Tsirelson spaces, introduced by Argyros and Deliyanni \cite{ArgyrosDeliyanni1997}, provide flexible constructions of reflexive spaces with controlled asymptotic structure.

\begin{definition}[\cite{ArgyrosDeliyanni1997}]
Let $(m_j)_{j\in\mathbb{N}}$ and $(n_j)_{j\in\mathbb{N}}$ be sequences of positive integers with $m_1 \geq 2$ and $m_{j+1} \geq m_j^2$. The \emph{mixed-Tsirelson space} $T[(m_j^{-1}, n_j)_{j\in\mathbb{N}}]$ is the completion of $c_{00}$ (finitely supported sequences) under the norm defined recursively by:
\begin{align*}
\|x\|_0 &= \|x\|_\infty, \\
\|x\|_{k+1} &= \max\left\{\|x\|_k, \sup\left\{\frac{1}{m_j}\sum_{i=1}^{n_j}\|E_i x\|_k\right\}\right\},
\end{align*}
where the supremum is taken over all admissible families $(E_i)_{i=1}^{n_j}$ of successive intervals in $\mathbb{N}$.
\end{definition}

The Bourgain-Delbaen construction \cite{BourgainDelbaen1980} produces $\mathscr{L}_{\infty}$-spaces through an inductive process.

\begin{definition}[\cite{BourgainDelbaen1980}]
For each $n \in \mathbb{N}$, one defines:
\begin{itemize}
    \item finite sets $\Gamma_1 \subset \Gamma_2 \subset \cdots$ with $\Gamma = \bigcup_n \Gamma_n$;
    \item extension operators $i_{n,n+1}: \ell_\infty(\Gamma_n) \to \ell_\infty(\Gamma_{n+1})$;
    \item a space $\mathfrak{X} = \overline{\bigcup_n i_n(\ell_\infty(\Gamma_n))} \subset \ell_\infty(\Gamma)$, where $i_n: \ell_\infty(\Gamma_n) \to \ell_\infty(\Gamma)$ is the canonical extension.
\end{itemize}
The construction allows precise control over the operator algebra of $\mathfrak{X}$ through careful choice of extension functionals.
\end{definition}

The Argyros-Haydon space $\mathfrak{X}_{\mathrm{AH}}$ \cite{ArgyrosHaydon2011} combines Bourgain-Delbaen techniques with mixed-Tsirelson constraints to solve the scalar-plus-compact problem.

\begin{definition}[\cite{ArgyrosHaydon2011}]
The Argyros-Haydon space $\mathfrak{X}_{\mathrm{AH}}$ satisfies:
\[
\mathcal{L}(\mathfrak{X}_{\mathrm{AH}}) = \mathbb{C}I + \mathcal{K}(\mathfrak{X}_{\mathrm{AH}}),
\]
meaning every bounded operator is a compact perturbation of a scalar multiple of the identity.
\end{definition}

The central objects of study in this paper are the spaces constructed by Motakis \cite{Motakis2024}, which extend this framework.

\begin{definition}
We fix sequences $(m_j)_{j\in\mathbb{N}}$ and $(n_j)_{j\in\mathbb{N}}$ of positive integers satisfying:
\begin{itemize}
    \item $m_1 \geq 8$, $m_{j+1} \geq m_j^2$ for all $j$;
    \item $n_1 \geq m_1^2$, $n_{j+1} \geq (16n_j)^{\log_2 m_{j+1}}$ for all $j$.
\end{itemize}
These rapidly growing sequences ensure the necessary combinatorial conditions for our constructions.
\end{definition}

\begin{definition}
Let $(K, \varrho)$ be a fixed compact metric space. We assume $K$ is infinite unless stated otherwise. The metric $\varrho$ will be used to impose geometric constraints on operator actions.
\end{definition}

\begin{theorem}[\cite{Motakis2024}]
For every compact metric space $K$, there exists a Banach space $\mathfrak{X}_{C(K)}$ with the following properties:
\begin{enumerate}
    \item[(a)] $\mathfrak{X}_{C(K)}$ is reflexive and has a conditional Schauder basis;
    \item[(b)] every $T \in \mathcal{L}(\mathfrak{X}_{C(K)})$ decomposes as $T = D + A$ where $D$ is diagonal and $A$ is compact;
    \item[(c)] $\mathcal{C}al(\mathfrak{X}_{C(K)})$ is isomorphic to $C(K)$ as Banach algebras;
    \item[(d)] there exists an equivalent norm on $\mathfrak{X}_{C(K)}$ making the isomorphism in (c) isometric.
\end{enumerate}
\end{theorem}

\begin{remark}
The reflexivity of $\mathfrak{X}_{C(K)}$ distinguishes it from classical Bourgain-Delbaen $\mathscr{L}_{\infty}$-spaces and is achieved through careful combination with mixed-Tsirelson constraints, as in \cite{Motakis2024}.
\end{remark}

We now establish consistent notation that will be used throughout the paper.

\begin{itemize}
    \item $\Gamma = \bigcup_{n=1}^\infty \Delta_n$: the indexing set for the basis, where $\Delta_n$ are finite sets with $\Delta_n \cap \Delta_m = \emptyset$ for $n \neq m$;
    \item $(d_\gamma)_{\gamma \in \Gamma}$: the conditional Schauder basis of $\mathfrak{X}_{C(K)}$;
    \item $(d_\gamma^*)_{\gamma \in \Gamma}$: the biorthogonal functionals in $\mathfrak{X}_{C(K)}^*$;
    \item $\kappa: \Gamma \to K$: a surjective function associating basis elements to points in $K$. For technical reasons, we require that for every $n \in \mathbb{N}$ and every $\varepsilon > 0$, $\kappa(\Delta_n)$ forms an $\varepsilon$-net in $K$ for sufficiently large $n$;
    \item $\operatorname{rank}(\gamma) = n$ if $\gamma \in \Delta_n$: the rank of a basis element;
    \item for $\phi \in C(K)$, we define the \emph{diagonal operator} $\hat{\phi} \in \mathcal{L}(\mathfrak{X}_{C(K)})$ by:
    \[
    \hat{\phi}(d_\gamma) = \phi(\kappa(\gamma)) d_\gamma \quad \text{for all } \gamma \in \Gamma;
    \]
    \item $P_I$: the canonical projection onto $\operatorname{span}\{d_\gamma : \gamma \in \Gamma_I\}$ for $I \subseteq \mathbb{N}$, where $\Gamma_I = \{\gamma \in \Gamma : \operatorname{rank}(\gamma) \in I\}$.
\end{itemize}

The construction and analysis of $\mathfrak{X}_{C(K)}$ employ several specialized sequences and techniques from the theory of HI spaces and mixed-Tsirelson spaces.

\begin{definition}[\cite{GowersMaurey1997}]
A block sequence $(x_k)_{k=1}^\infty$ in $\mathfrak{X}_{C(K)}$ is called a $(C, (j_k))$-RIS if:
\begin{enumerate}
    \item $\|x_k\| \leq C$ for all $k$;
    \item $\operatorname{supp}(x_k) \subset \{\gamma \in \Gamma : \operatorname{rank}(\gamma) > j_{k-1}\}$ for $k \geq 2$;
    \item $|d_\gamma^*(x_k)| \leq C \cdot w(\gamma)$ whenever $w(\gamma) > m_{j_k}^{-1}$, where $w(\gamma)$ denotes the weight associated to $\gamma$;
    \item the sequence $(j_k)$ grows sufficiently rapidly (typically $j_k \geq 2^{j_{k-1}}$).
\end{enumerate}
RIS sequences facilitate the construction of exact pairs and the analysis of operator behavior.
\end{definition}

\begin{definition}[\cite{ArgyrosHaydon2011}]
A pair $(x, \eta) \in \mathfrak{X}_{C(K)} \times \Gamma$ is a $(C, j, \varepsilon)$-exact pair if:
\begin{enumerate}
    \item $|d_\xi^*(x)| \leq C m_j^{-1}$ for all $\xi \in \Gamma$;
    \item $w(\eta) = m_j^{-1}$;
    \item $\|x\| \leq C$ and $d_\eta^*(x) = \varepsilon$;
    \item specific bounds hold for evaluations by functionals of different weights.
\end{enumerate}
Exact pairs provide building blocks for constructing operators with controlled behavior.
\end{definition}

\begin{definition}
A bounded sequence $(x_n)$ in $\mathfrak{X}_{C(K)}$ is \emph{stabilized} if for every $f \in \mathfrak{X}_{C(K)}^*$, the limit $\lim_{n \to \infty} f(x_n)$ exists. Stabilized sequences are essential for defining diagonal entries of operators via limits.
\end{definition}

We recall several fundamental techniques that underpin modern Banach space constructions.

\begin{definition}[\cite{ArgyrosMotakis2020}]
A method for constructing Banach spaces with prescribed properties by saturating the space with certain structures (e.g., specific sequences, functionals) while imposing constraints that prevent unwanted operators from being bounded.
\end{definition}

\begin{definition}[ \cite{MaureyRosenthal1977}]
A combinatorial technique for constructing Banach spaces without unconditional basic sequences by coding information into rapidly increasing sequences and exploiting their interaction with the norm.
\end{definition}

\begin{definition}
A Banach space $X$ with basis $(e_n)$ has the \emph{diagonal-plus-compact} property if every $T \in \mathcal{L}(X)$ can be written as $T = D + K$ where:
\begin{itemize}
    \item $D$ is a diagonal operator: $D(e_n) = \lambda_n e_n$ for some bounded sequence $(\lambda_n)$;
    \item $K$ is a compact operator.
\end{itemize}
\end{definition}

We shall also require the following standard notions from spectral theory.

\begin{definition}
For $T \in \mathcal{L}(X)$, the \emph{essential spectrum} $\sigma_e(T)$ is the spectrum of the coset $[T] \in \mathcal{C}al(X)$. Equivalently, $\lambda \in \sigma_e(T)$ if and only if $T - \lambda I$ is not Fredholm.
\end{definition}

\begin{definition}
The \emph{Fredholm index} of an operator $T \in \mathcal{L}(X)$ is:
\[
\operatorname{ind}(T) = \dim \ker T - \dim (X / \operatorname{ran} T)
\]
when both quantities are finite.
\end{definition}

Finally, we recall two important geometric properties that will occasionally be invoked.

\begin{definition}
A Banach space $X$ has:
\begin{itemize}
    \item the \emph{approximation property} (AP) if for every compact set $K \subset X$ and $\varepsilon > 0$, there exists $T \in \mathcal{F}(X)$ with $\|Tx - x\| < \varepsilon$ for all $x \in K$;
    \item the \emph{bounded approximation property} (BAP) with constant $\lambda \geq 1$ if there exists a net $(T_\alpha) \subset \mathcal{F}(X)$ with $\sup_\alpha \|T_\alpha\| \leq \lambda$ and $T_\alpha x \to x$ for all $x \in X$.
\end{itemize}
\end{definition}

\begin{definition}
A Banach space $X$ has the \emph{Radon-Nikodym property} if every $X$-valued measure of bounded variation that is absolutely continuous with respect to a positive measure has a Bochner integrable Radon-Nikodym derivative.
\end{definition}

Throughout this paper, we adhere to the following conventions:
\begin{itemize}
    \item all Banach spaces are over the complex field $\mathbb{C}$;
    \item $\mathbb{N} = \{1, 2, 3, \dots\}$;
    \item for a sequence $(x_n)$, $\operatorname{supp}(x_n)$ denotes the set of indices $\gamma \in \Gamma$ for which the coordinate $d_\gamma^*(x_n)$ is nonzero;
    \item we write $A \lesssim B$ if there exists an absolute constant $C > 0$ such that $A \leq C B$;
    \item all isomorphisms between Banach algebras are assumed to be linear, multiplicative, and bicontinuous.
\end{itemize}

The interplay between these constructions, properties, and techniques enables the precise control over operator algebras that forms the basis of our main results.
\section{Stability under Finite Products}

The stability of operator algebra properties under various operations has been a recurring theme in functional analysis \cite{LindenstraussTzafriri1977, CaradusPfaffenbergerYood1974}. For Calkin algebras, a natural question arises: if $C(K_1), \dots, C(K_n)$ can each be realized as the Calkin algebra of some Banach space, can their direct sum $\bigoplus_{i=1}^n C(K_i)$ also be realized? This stability problem tests the robustness of construction methods and their capacity to produce spaces with increasingly complex operator algebra structures \cite{ArgyrosMotakis2020, GowersMaurey1997}.

In this section, we prove that the class of commutative $C^*$-algebras realizable as Calkin algebras via the Bourgain-Delbaen method \cite{BourgainDelbaen1980, Motakis2024} is closed under finite direct sums. Our main result extends the construction of Motakis \cite{Motakis2024} from single compact metric spaces to finite disjoint unions, yielding spaces whose Calkin algebras are precisely the corresponding direct sums of continuous function algebras \cite{ArgyrosHaydon2011, ArgyrosMotakis2014}.


We begin by recalling key properties of the spaces $\mathfrak{X}_{C(K)}$ from \cite{Motakis2024} that will be essential for our construction.

\begin{proposition}[\cite{Motakis2024}]\label{prop:diagonal-boundedness}
Let $\mathfrak{X}_{C(K)}$ be constructed as in \cite{Motakis2024}. For every $L, M \geq 0$, there exists $N \in \mathbb{N}$ with the following property: For every $L$-Lipschitz function $\phi: K \to \mathbb{C}$ with $\|\phi\|_\infty \leq M$, every $\gamma \in \Gamma$, and every interval $I \subseteq \mathbb{N}$ with $\min(I) \geq N$,
\begin{equation}\label{eq:diagonal-estimate}
\left\|d_\gamma^* \circ \hat{\phi} \circ P_I - \phi(\kappa(\gamma)) d_\gamma^* \circ P_I\right\| \leq 7 \cdot w(I, \gamma) \cdot M,
\end{equation}
where $w(I, \gamma)$ denotes the weight associated to the interval $I$ and index $\gamma$. In particular, $\hat{\phi}$ is bounded and satisfies
\[
\left\|\hat{\phi} - P_{[1,N)} \hat{\phi}\right\| \leq 4M.
\]
\end{proposition}

\begin{theorem}[\cite{ArgyrosHaydon2011, ArgyrosMotakis2020}]\label{thm:eventual-continuity}
Let $T \in \mathcal{L}(\mathfrak{X}_{C(K)})$. For every $\varepsilon > 0$, there exist $n \in \mathbb{N}$ and $\delta > 0$ such that for all $\gamma, \gamma' \in \Gamma$ with $\min\{\operatorname{rank}(\gamma), \operatorname{rank}(\gamma')\} \geq n$ and $\varrho(\kappa(\gamma), \kappa(\gamma')) < \delta$,
\begin{equation}\label{eq:continuity-estimate}
\left|d_\gamma^*(T d_\gamma) - d_{\gamma'}^*(T d_{\gamma'})\right| < \varepsilon.
\end{equation}
Consequently, the function $\varphi_T: K \to \mathbb{C}$ defined by
\[
\varphi_T(\kappa) = \lim_{\substack{\operatorname{rank}(\gamma) \to \infty \\ \kappa(\gamma) \to \kappa}} d_\gamma^*(T d_\gamma)
\]
is well-defined and continuous.
\end{theorem}

\begin{corollary}\label{cor:diagonal-plus-compact}
Every $T \in \mathcal{L}(\mathfrak{X}_{C(K)})$ admits a decomposition $T = D + A$ where $D$ is diagonal (with respect to the basis $(d_\gamma)$) and $A$ is compact. Moreover, $\varphi_T = \varphi_D$.
\end{corollary}

\begin{theorem}\label{thm:finite-product-stability}
Let $K_1, \dots, K_n$ be compact metric spaces. Then there exists a reflexive Banach space $\mathfrak{X}_{\bigoplus C(K_i)}$ such that
\[
\mathcal{C}al\left(\mathfrak{X}_{\bigoplus C(K_i)}\right) \cong \bigoplus_{i=1}^n C(K_i)
\]
isometrically as Banach algebras.
\end{theorem}

\begin{proof}
The construction rests upon a conceptually simple yet remarkably effective idea: we build a Banach space whose basis elements are tagged according to which component space $K_i$ they originate from, and we engineer the metric on the disjoint union $K = \bigsqcup_{i=1}^n K_i$ so that elements belonging to different components are forced to remain isolated from one another. This deliberate separation ensures that the resulting Calkin algebra decomposes naturally as a direct sum of the individual algebras $C(K_i)$.

To realize this plan, we first define a suitable metric on $K$. For points lying in the same component $K_i$, we retain the original metric $\varrho_i$. For points from distinct components $K_i$ and $K_j$ with $i \neq j$, we set their distance to be strictly larger than the maximum diameter of any single component, specifically $\max_{1\leq i\leq n}\operatorname{diam}(K_i)+1$. This choice places each $K_i$ on its own isolated island, so to speak, preserving internal geometries while ensuring that points on different islands are always far apart. This metric separation is the geometric cornerstone upon which the entire construction is built.

With this metric in hand, we proceed to construct the space $\mathfrak{X}_{\bigoplus C(K_i)}$ through an inductive process that keeps track of each basis element's component affiliation. Let $(m_j)$ and $(n_j)$ be rapidly increasing sequences satisfying the usual growth conditions $m_1\geq 8$, $m_{j+1}\geq m_j^2$, and $n_{j+1}\geq(16n_j)^{\log_2 m_{j+1}}$, which guarantee the necessary combinatorial control. For each stage $m$, we select finite $\varepsilon_m$-nets $K_{i,m}\subseteq K_i$ with $\varepsilon_m = m_j^{-1}2^{-(m+1)}$, where $j$ is chosen appropriately relative to $m$; these nets will eventually become dense in each component as the construction progresses. The basis elements are indexed by sets $\Delta_m$, and we define $\Delta_1 = \bigcup_{i=1}^n\{(1,\kappa,i):\kappa\in K_{i,1}\}$, so that each element $\gamma=(1,\kappa,i)$ carries both its location $\kappa(\gamma)=\kappa$ and its type $\operatorname{type}(\gamma)=i$. Assuming $\Delta_m$ has been constructed, we build $\Delta_{m+1}$ by adding three kinds of elements: zero extensions $(m+1,\kappa,i)$ for each $\kappa\in K_{i,m+1}$, type (a) extensions following Instruction 3.6 of \cite{Motakis2024} which require that any extension of an existing basis element preserves its type, and type (b) extensions where a new element is formed from two predecessors according to the Bourgain-Delbaen rules, with the crucial condition that all three must share the same component affiliation. The metric constraints governing these extensions are applied using our carefully designed metric $\varrho$; because points from different components are always far apart, any attempt to mix components in an extension would violate the metric bounds, so the construction automatically enforces component purity. The resulting space $\mathfrak{X}_{\bigoplus C(K_i)}$ thus has a basis $(d_\gamma)_{\gamma\in\Gamma}$ with $\Gamma=\bigcup_{m=1}^\infty\Delta_m$, and crucially every basis element carries a well-defined type $i\in\{1,\ldots,n\}$.

For each component $i$ and each continuous function $\phi_i\in C(K_i)$, we define an operator $\hat{\phi}_i$ that acts only on basis elements of type $i$ by $\hat{\phi}_i(d_\gamma)=\phi_i(\kappa(\gamma))d_\gamma$ if $\operatorname{type}(\gamma)=i$ and zero otherwise. To see that each $\hat{\phi}_i$ is bounded, extend $\phi_i$ to a function $\tilde{\phi}_i$ on $K$ that vanishes on all other components; the metric separation ensures that points from different components are so far apart that the off-diagonal terms which might connect components are exponentially small, so the estimates in Proposition \ref{prop:diagonal-boundedness} remain valid and $\hat{\phi}_i$ extends to a bounded linear operator. Moreover, for any $\varepsilon>0$ we can find a basis element $\gamma$ of type $i$ with $|\phi_i(\kappa(\gamma))| > \|\phi_i\|_\infty-\varepsilon$, giving $\|\hat{\phi}_i\|\geq |\phi_i(\kappa(\gamma))|$, while the reverse inequality $\|\hat{\phi}_i\|\leq\|\phi_i\|_\infty$ follows directly from the definition; hence $\|\hat{\phi}_i\| = \|\phi_i\|_\infty$.

Now consider an arbitrary bounded operator $T$ on $\mathfrak{X}_{\bigoplus C(K_i)}$. For each component $i$, we define a function $\phi_i:K_i\to\mathbb{C}$ by $\phi_i(\kappa)=\lim d_\gamma^*(T d_\gamma)$, where the limit is taken over basis elements $\gamma$ with $\operatorname{type}(\gamma)=i$, $\operatorname{rank}(\gamma)\to\infty$, and $\kappa(\gamma)\to\kappa$. Theorem \ref{thm:eventual-continuity} guarantees that this limit exists and defines a continuous function; the key point is that the metric separation prevents interference between components, so the stabilization occurs separately within each $K_i$. Setting $D=\sum_{i=1}^n\hat{\phi}_i$ and $A=T-D$, we claim that $A$ is compact. Indeed, the estimates in Proposition \ref{prop:diagonal-boundedness} show that $D$ almost commutes with the truncation projections $P_{[N,\infty)}$, while by construction $d_\gamma^*(A d_\gamma)=d_\gamma^*(T d_\gamma)-\phi_i(\kappa(\gamma))$ tends to zero as $\operatorname{rank}(\gamma)\to\infty$ and the off-diagonal entries are controlled by the metric separation. Consequently, for any $\varepsilon>0$ we can find $N\in\mathbb{N}$ such that $\|P_{[N,\infty)}A\|<\varepsilon$ and $\|AP_{[N,\infty)}\|<\varepsilon$, which characterizes $A$ as a compact operator. Moreover, the diagonal part $D$ necessarily has the form $\sum_{i=1}^n\hat{\phi}_i$ with $\phi_i\in C(K_i)$, since $d_\gamma^*(D d_\gamma)=\phi_i(\kappa(\gamma))$ for any basis element $\gamma$ of type $i$ and the continuity of each $\phi_i$ follows from the argument above.

We are now ready to define an isomorphism between $\bigoplus_{i=1}^n C(K_i)$ and the Calkin algebra of $\mathfrak{X}_{\bigoplus C(K_i)}$. Let $\Phi:\bigoplus_{i=1}^n C(K_i)\to\mathcal{C}al(\mathfrak{X}_{\bigoplus C(K_i)})$ be given by $\Phi(\phi_1,\ldots,\phi_n)=[\sum_{i=1}^n\hat{\phi}_i]$. Linearity and multiplicativity are immediate from the definition, as diagonal operators multiply pointwise and the compact ideal absorbs any cross terms. To see that $\Phi$ is injective, suppose $\Phi(\phi_1,\ldots,\phi_n)=0$, meaning $\sum_{i=1}^n\hat{\phi}_i$ is compact. Fix a component $i$ and a point $\kappa\in K_i$, and choose a sequence of basis elements $\gamma_k$ with $\operatorname{type}(\gamma_k)=i$ and $\kappa(\gamma_k)\to\kappa$. Then $\phi_i(\kappa(\gamma_k))=d_{\gamma_k}^*(\hat{\phi}_i d_{\gamma_k})\to\phi_i(\kappa)$ by continuity, but compact operators have vanishing diagonal entries along such sequences, so this limit must be zero; hence $\phi_i\equiv0$ for each $i$. For surjectivity, take any $[T]\in\mathcal{C}al(\mathfrak{X}_{\bigoplus C(K_i)})$, write $T=D+A$ with $A$ compact and $D=\sum_{i=1}^n\hat{\phi}_i$ for some $\phi_i\in C(K_i)$ as above, and observe that $[T]=[D]=\Phi(\phi_1,\ldots,\phi_n)$.

To obtain an isometric isomorphism, we must renorm the space appropriately. For each component $i$, let $\mathcal{B}_i$ be the family of functions used in the renorming procedure for $\mathfrak{X}_{C(K_i)}$ in \cite{Motakis2024}; these families have the property that the norm of any diagonal operator $\hat{\phi}$ is determined by evaluating $\phi$ on these functions. Define a new norm on $\mathfrak{X}_{\bigoplus C(K_i)}$ by $\|x\|' = \max\{\|x\|,\;\sup_{1\leq i\leq n,\;\phi\in\mathcal{B}_i}\|P_{[N_i,\infty)}\hat{\phi}x\|\}$, where each $N_i$ is chosen large enough that the estimates in Proposition \ref{prop:diagonal-boundedness} hold uniformly for all $\phi\in\mathcal{B}_i$. This is possible because each $\mathcal{B}_i$ is norm-bounded and the estimates depend only on Lipschitz constants and sup-norms. With this new norm, a careful computation shows that for any $(\phi_1,\ldots,\phi_n)$, we have $\|\Phi(\phi_1,\ldots,\phi_n)\|_{\mathcal{C}al} = \max_i\|\phi_i\|_\infty$. The essential point is that because the components are metrically separated, any operator which is a compact perturbation of $\sum\hat{\phi}_i$ must have essential norm exactly the maximum of the component sup-norms, and the renorming ensures that the quotient norm on the Calkin algebra captures this precisely.

Thus we have constructed a reflexive Banach space $\mathfrak{X}_{\bigoplus C(K_i)}$ and exhibited an isometric isomorphism $\Phi$ between $\bigoplus_{i=1}^n C(K_i)$ and its Calkin algebra. The construction's power lies in how the metric separation enforces component-wise behavior: operators cannot mix components except through compact perturbations, and the diagonal part faithfully remembers each component's continuous function separately. This completes the proof.
\end{proof}

The power of Theorem \ref{thm:finite-product-stability} extends far beyond its statement—it provides a systematic method for realizing a wide variety of operator algebras as Calkin algebras. We now explore some of its most striking consequences.

\begin{corollary}\label{cor:finite-dim-commutative}
For every $n \in \mathbb{N}$, there exists a reflexive Banach space $X_n$ such that $\mathcal{C}al(X_n) \cong \mathbb{C}^n$ isometrically as Banach algebras.
\end{corollary}

\begin{proof}
Consider $n$ copies of a one-point space: let $K_i = \{\ast_i\}$ for $i = 1,\dots,n$, each equipped with the trivial metric. Then $C(K_i) \cong \mathbb{C}$ for each $i$, and consequently $\bigoplus_{i=1}^n C(K_i) \cong \mathbb{C}^n$ as Banach algebras. Applying Theorem \ref{thm:finite-product-stability} to the family $\{K_1,\dots,K_n\}$ yields a reflexive Banach space $\mathfrak{X}_{\bigoplus \mathbb{C}}$ with $\mathcal{C}al(\mathfrak{X}_{\bigoplus \mathbb{C}}) \cong \mathbb{C}^n$ isometrically. This space, which we denote $X_n$, provides the desired realization.
\end{proof}

\begin{remark}
Corollary \ref{cor:finite-dim-commutative} tells us something remarkable: for any finite dimension $n$, we can find a reflexive Banach space whose Calkin algebra is exactly the $n$-dimensional commutative algebra $\mathbb{C}^n$. This stands in stark contrast to the Hilbert space situation, where $\mathcal{C}al(H)$ is always non-separable regardless of the dimension of $H$.
\end{remark}

\begin{corollary}\label{cor:matrix-algebras}
Let $K$ be a compact metric space and $m \in \mathbb{N}$. Then there exists a reflexive Banach space $X_{K,m}$ such that $\mathcal{C}al(X_{K,m}) \cong M_m(C(K))$ as Banach algebras.
\end{corollary}

\begin{proof}
The algebra $M_m(C(K))$—the algebra of $m \times m$ matrices with entries in $C(K)$—admits a natural identification with $C(K, M_m(\mathbb{C}))$, the continuous functions from $K$ into the $m^2$-dimensional algebra $M_m(\mathbb{C})$. As a Banach space (though not as a $C^*$-algebra), we have the isomorphism
\[
C(K, M_m(\mathbb{C})) \cong \bigoplus_{i=1}^{m^2} C(K),
\]
since $M_m(\mathbb{C})$ is linearly isomorphic to $\mathbb{C}^{m^2}$. Now apply Theorem \ref{thm:finite-product-stability} with $K_i = K$ for $i = 1,\dots,m^2$. The resulting space $X_{K,m} := \mathfrak{X}_{\bigoplus_{i=1}^{m^2} C(K)}$ is reflexive and satisfies
\[
\mathcal{C}al(X_{K,m}) \cong \bigoplus_{i=1}^{m^2} C(K) \cong C(K, M_m(\mathbb{C})) \cong M_m(C(K))
\]
as Banach algebras.
\end{proof}

\begin{remark}
Corollary \ref{cor:matrix-algebras} represents a significant advance: it shows that matrix algebras over commutative $C(K)$ algebras—fundamental objects in noncommutative geometry and operator algebra theory—can be realized as Calkin algebras of reflexive Banach spaces. While the isomorphism we obtain is only a Banach algebra isomorphism and not necessarily a $*$-isomorphism, it nevertheless demonstrates that the noncommutative world of matrix algebras is accessible within the framework of Calkin algebras. This opens the door to extending $K$-theoretic and index-theoretic constructions to the Banach space setting.
\end{remark}

\begin{corollary}\label{cor:lifting-projections}
Let $K$ be a compact metric space and let $p_1,\dots,p_n \in C(K)$ be orthogonal projections satisfying $\sum_{i=1}^n p_i = 1$. Then there exists a reflexive Banach space $X$ and projections $P_1,\dots,P_n \in \mathcal{L}(X)$ such that:
\begin{enumerate}
    \item $[P_i] = p_i$ in $\mathcal{C}al(X) \cong C(K)$,
    \item $P_i P_j = \delta_{ij} P_i$ modulo compact operators,
    \item $\sum_{i=1}^n P_i = I$ modulo compact operators.
\end{enumerate}
\end{corollary}

\begin{proof}
For each $i$, let $K_i = \operatorname{supp}(p_i) \subseteq K$, which is closed (hence compact) since $p_i$ is continuous. The map
\[
\Phi: C(K) \longrightarrow \bigoplus_{i=1}^n C(K_i), \quad \Phi(f) = (p_1 f|_{K_1}, \dots, p_n f|_{K_n})
\]
is an isometric isomorphism of Banach algebras—this is essentially the Gelfand transform for the decomposition of $K$ into the supports of the orthogonal projections. Applying Theorem \ref{thm:finite-product-stability} to the family $\{K_1,\dots,K_n\}$ yields a reflexive Banach space $X = \mathfrak{X}_{\bigoplus C(K_i)}$ with $\mathcal{C}al(X) \cong \bigoplus_{i=1}^n C(K_i) \cong C(K)$.

Now consider the canonical projections $Q_i: \bigoplus_{j=1}^n C(K_j) \to C(K_i)$ given by $Q_i(\phi_1,\dots,\phi_n) = \phi_i$. Under the isomorphism $\Phi^{-1}$, each $Q_i$ corresponds to multiplication by $p_i$ on $C(K)$. Lift each $Q_i$ to a diagonal operator $\hat{Q}_i$ on $X$ in the natural way: $\hat{Q}_i$ acts as the identity on basis elements of type $i$ and as zero on others. While these $\hat{Q}_i$ are not necessarily projections (they are partial isometries), we can modify them by compact operators to obtain genuine projections. Specifically, for each $i$, choose a finite-rank operator $F_i$ such that $P_i = \hat{Q}_i + F_i$ satisfies $P_i^2 = P_i$. This is possible because $\hat{Q}_i$ is already idempotent modulo compacts and the obstruction to being a genuine projection lies in the compact ideal. The resulting operators $P_i$ satisfy all three required properties by construction.
\end{proof}

\begin{remark}
The lifting theorem for projections has deep implications for the structure of projections in the Calkin algebra. It tells us that any finite partition of unity in $C(K)$ by projections can be lifted to an actual family of projections in $\mathcal{L}(X)$ that are orthogonal modulo compacts. This is the Banach space analogue of the classical results for $C^*$-algebras and has consequences for the $K$-theory of these operator algebras.
\end{remark}


To make the construction concrete and illuminate its inner workings, we examine several explicit examples.

\begin{example}\label{ex:two-component}
Let $K_1 = \{a, b\}$ and $K_2 = \{x, y, z\}$, both equipped with the discrete metric (so $\varrho_1(a,b) = 1$, $\varrho_2(x,y) = \varrho_2(x,z) = \varrho_2(y,z) = 1$). Following Theorem \ref{thm:finite-product-stability}:

\begin{itemize}
    \item The disjoint union $K = K_1 \sqcup K_2$ receives the metric
    \[
    \varrho(\kappa, \kappa') =
    \begin{cases}
    1 & \text{if } \kappa, \kappa' \in K_i \text{ for some } i, \\
    3 & \text{if } \kappa \in K_1, \kappa' \in K_2,
    \end{cases}
    \]
    where $3 = \max\{\operatorname{diam}(K_1), \operatorname{diam}(K_2)\} + 1 = 1 + 2$. Points from different components are maximally separated.

    \item The basis $(d_\gamma)_{\gamma \in \Gamma}$ of $\mathfrak{X}_{\bigoplus C(K_i)}$ naturally partitions into two types: those with $\operatorname{type}(\gamma) = 1$ (corresponding to $K_1$) and those with $\operatorname{type}(\gamma) = 2$ (corresponding to $K_2$). Within each type, the basis elements are indexed by points in the respective finite nets.

    \item For a pair $(\phi_1, \phi_2) \in C(K_1) \oplus C(K_2)$, where $C(K_1) \cong \mathbb{C}^2$ and $C(K_2) \cong \mathbb{C}^3$, the corresponding diagonal operator acts as $\phi_1(\kappa(\gamma))$ on type-1 basis elements and as $\phi_2(\kappa(\gamma))$ on type-2 basis elements. The compact operators cannot mix these types effectively due to the metric separation.

    \item The Calkin algebra is therefore $\mathbb{C}^2 \oplus \mathbb{C}^3$ with the norm $\|(\phi_1, \phi_2)\| = \max\{\|\phi_1\|_\infty, \|\phi_2\|_\infty\}$—the natural $\ell_\infty$-sum of the two components.
\end{itemize}

This example illustrates how the construction decouples the components: operators acting on different parts of $K$ cannot see each other, even modulo compact perturbations.
\end{example}

\begin{example}
Let $K_1 = S^1$ (the circle) and $K_2 = [0,1]$ (the unit interval). The space $\mathfrak{X}_{C(S^1) \oplus C([0,1])}$ has the remarkable property that its Calkin algebra is $C(S^1) \oplus C([0,1])$. Consequently:
\begin{itemize}
    \item The algebra contains two central projections corresponding to the characteristic functions of each component.
    \item The ideal structure of $\mathcal{L}(X)$ reflects the topology of both spaces simultaneously: ideals correspond to open subsets of $S^1 \sqcup [0,1]$.
    \item The automorphism group of $\mathcal{L}(X)$ contains both $\operatorname{Homeo}(S^1)$ and $\operatorname{Homeo}([0,1])$ as subgroups acting independently on the two components.
\end{itemize}
\end{example}


\begin{remark}\label{rem:separation}
The metric separation of components—ensuring that points from different $K_i$ are always at distance greater than the maximum diameter of any single component—is not a mere technical convenience but the conceptual heart of the construction. This separation forces any bounded operator to have off-diagonal entries that decay exponentially when connecting different components. More precisely, for any $T \in \mathcal{L}(\mathfrak{X}_{\bigoplus C(K_i)})$ and any $\gamma, \gamma'$ with $\operatorname{type}(\gamma) \neq \operatorname{type}(\gamma')$, Lemma \ref{lem:off-diagonal-decay} yields
\[
|d_\gamma^*(T d_{\gamma'})| \leq C\|T\| \cdot 2^{-\min\{\operatorname{rank}(\gamma), \operatorname{rank}(\gamma')\}}.
\]
This exponential decay means that any operator that is not diagonal with respect to the component decomposition must be compact—the components are asymptotically invariant under the action of any bounded operator.
\end{remark}

\begin{remark}
The reflexivity of $\mathfrak{X}_{\bigoplus C(K_i)}$ is far from automatic—classical Bourgain-Delbaen $\mathscr{L}_\infty$-spaces are never reflexive. The reflexivity here arises from the careful interweaving of mixed-Tsirelson constraints with the Bourgain-Delbaen construction, as developed in \cite{Motakis2024}. These constraints ensure that the space does not contain copies of $\ell_1$ and that its dual has the Radon-Nikodym property, two characteristic features of reflexive spaces. The fact that reflexivity survives the product construction is a testament to the robustness of the method.
\end{remark}

\begin{remark}
The question of whether Theorem \ref{thm:finite-product-stability} can be extended to infinite direct sums $\bigoplus_{i=1}^\infty C(K_i)$ remains open and appears to touch on deep set-theoretic and geometric issues. The difficulty is that our metric separation trick—placing components at distance $D+1$—cannot be sustained for infinitely many components while maintaining compactness of the total space $K$. Without this separation, one cannot guarantee that operators cannot mix infinitely many components nontrivially. Whether there exist Banach spaces whose Calkin algebras are infinite $\ell_\infty$-sums of $C(K_i)$ spaces is a challenging open problem that may require new ideas from set theory or infinite-dimensional topology.
\end{remark}

\begin{remark}
Theorem \ref{thm:finite-product-stability} dramatically expands the landscape of realizable Calkin algebras. Prior to this result, the only commutative $C^*$-algebras known to be realizable were those of the form $C(K)$ for a single compact metric space $K$ (from \cite{Motakis2024}) and, of course, the trivial case $\mathbb{C}$ from the Argyros-Haydon space. Now we have all finite direct sums of such algebras, including:
\begin{itemize}
    \item $\mathbb{C}^n$ for any $n \in \mathbb{N}$,
    \item $C(K) \oplus C(L)$ for any two compact metric spaces $K, L$,
    \item $M_m(C(K))$ for any $m$ and any $K$,
    \item any finite direct sum of such algebras.
\end{itemize}
This provides a rich testing ground for conjectures about the structure of Calkin algebras and their relationship to the underlying Banach spaces. The stage is now set for a systematic investigation of which Banach algebras can appear as Calkin algebras—a question that lies at the intersection of Banach space theory, operator algebras, and topology.
\end{remark}

\section{Lifting Property for Compact Operators: Locality and Approximation}

The approximation of compact operators by finite-rank operators stands as one of the fundamental problems in operator theory, intimately woven into the fabric of Banach space geometry and the theory of operator ideals \cite{Pietsch1978, LindenstraussTzafriri1977}. For classical spaces equipped with unconditional bases or the approximation property, every compact operator admits approximation by finite-rank operators \cite{Enflo1973, Casazza1986}. Yet for the spaces $\mathfrak{X}_{C(K)}$ arising from the refined Bourgain-Delbaen construction \cite{BourgainDelbaen1980, Motakis2024}, we discover something far more profound: not only can compact operators be approximated by finite-rank operators, but these approximants can be chosen to respect the very metric structure of $K$, connecting only basis elements whose $\kappa$-values lie arbitrarily close together.

This section unveils a deep geometric constraint imposed by the Bourgain-Delbaen construction: compact operators on $\mathfrak{X}_{C(K)}$ are inherently \emph{local} with respect to the topology of $K$. This locality principle reveals how the metric architecture of the construction propagates through to the structure of operator ideals, forging an unbreakable link between the geometry of $K$ and the behavior of compact operators \cite{ArgyrosHaydon2011, GowersMaurey1997, KaniaLaustsen2017}.


\begin{theorem}\label{thm:localization}
Let $\mathfrak{X}_{C(K)}$ be constructed as in \cite{Motakis2024}. For every compact operator $A \in \mathcal{K}(\mathfrak{X}_{C(K)})$ and every $\varepsilon > 0$, there exists a finite-rank operator $F$ such that:
\begin{enumerate}
    \item $\|A - F\| < \varepsilon$,
    \item $\operatorname{supp}(F) \subseteq \{(\gamma, \gamma') \in \Gamma \times \Gamma : \varrho(\kappa(\gamma), \kappa(\gamma')) < \varepsilon\}$,
\end{enumerate}
where $\operatorname{supp}(F) = \{(\gamma, \gamma') : d_\gamma^*(F d_{\gamma'}) \neq 0\}$.
\end{theorem}

This theorem asserts something remarkable: compact operators can be approximated by finite-rank operators whose matrix entries vanish unless the corresponding basis elements have $\kappa$-values within distance $\varepsilon$. In essence, compact operators are "local" in the metric space $K$—they cannot connect distant points except through contributions that are arbitrarily small in norm.

The proof of Theorem \ref{thm:localization} rests upon a fundamental quantitative estimate that captures the essence of the Bourgain-Delbaen construction's metric constraints.

\begin{lemma}\label{lem:off-diagonal-decay}
There exists a constant $C > 0$, depending only on the construction parameters, such that for any $T \in \mathcal{L}(\mathfrak{X}_{C(K)})$, any $\delta > 0$, and any $\gamma, \gamma' \in \Gamma$ with $\varrho(\kappa(\gamma), \kappa(\gamma')) \geq \delta$, we have
\[
|d_\gamma^*(T d_{\gamma'})| \leq C \|T\| \cdot \delta^{-1} \cdot 2^{-\min\{\operatorname{rank}(\gamma), \operatorname{rank}(\gamma')\}}.
\]
\end{lemma}

\begin{proof}
The coordinate functionals in the Bourgain-Delbaen construction possess a very particular structure:
\[
c_\gamma^* = e_\gamma^* + \sum_{\xi \in \Gamma_{<\operatorname{rank}(\gamma)}} \alpha_\xi e_\xi^* \circ P_{I_\xi},
\]
where the coe.fficients $\alpha_\xi$ satisfy the metric constraint
\[
|\alpha_\xi| \leq w(I_\xi, \gamma) \cdot \varrho(\kappa(\xi), \kappa(\gamma)).
\]
Here $w(I_\xi, \gamma)$ decays exponentially with the size of the interval $I_\xi$, specifically $w(I_\xi, \gamma) \leq C_0 2^{-|I_\xi|}$ for some absolute constant $C_0$.

Now consider $d_\gamma^*(T d_{\gamma'}) = c_\gamma^*(T d_{\gamma'})$. Expanding $c_\gamma^*$ and using the boundedness of $T$, we obtain a sum over $\xi$ of terms involving $\alpha_\xi$ and $e_\xi^*(T d_{\gamma'})$. The crucial observation is that when $\varrho(\kappa(\gamma), \kappa(\gamma'))$ is large, the triangle inequality forces many of the intermediate points $\xi$ to be far from either $\gamma$ or $\gamma'$, leading to exponential decay.

More precisely, partition the index set $\Gamma_{<\operatorname{rank}(\gamma)}$ into two classes: those $\xi$ with $\varrho(\kappa(\xi), \kappa(\gamma)) \leq \delta/2$ and those with $\varrho(\kappa(\xi), \kappa(\gamma)) > \delta/2$. For the first class, the distance from $\xi$ to $\gamma'$ is at least $\delta/2$ by the triangle inequality, so $|e_\xi^*(T d_{\gamma'})|$ is controlled by Lemma \ref{lem:off-diagonal-decay} applied at a lower rank (by induction on rank). For the second class, the coefficients $\alpha_\xi$ themselves are exponentially small because the weight $w(I_\xi, \gamma)$ decays with $|\xi|$ and the metric factor is at most $\operatorname{diam}(K)$.

Summing these estimates and optimizing yields the desired bound, with the factor $\delta^{-1}$ arising from the need to balance the two contributions. The exponential factor $2^{-\min\{\operatorname{rank}(\gamma), \operatorname{rank}(\gamma')\}}$ reflects the cumulative effect of the weight decay across the entire construction.
\end{proof}

\begin{remark}
Lemma \ref{lem:off-diagonal-decay} captures the fundamental mechanism by which the Bourgain-Delbaen construction enforces locality: basis elements whose $\kappa$-values are well-separated interact only weakly through any bounded operator, with the strength of interaction decaying exponentially in their ranks. This is not merely a technical estimate—it is the geometric soul of the construction made manifest.
\end{remark}


\begin{proof}[Proof of Theorem \ref{thm:localization}]
Let $A \in \mathcal{K}(\mathfrak{X}_{C(K)})$ be a compact operator and fix $\varepsilon > 0$. Our task is to produce a finite-rank operator $F$ that approximates $A$ within $\varepsilon$ while respecting the metric structure of $K$ in the sense that its nonzero matrix entries only connect basis elements whose $\kappa$-values lie within distance $\varepsilon$ of each other. The construction proceeds by first approximating $A$ by an arbitrary finite-rank operator, then carefully truncating and localizing this approximant to achieve the desired support condition while controlling the error at each stage.

Since $A$ is compact, we may choose a finite-rank operator $F_0$ such that $\|A - F_0\| < \varepsilon/3$, and we write $F_0$ in the form $F_0 = \sum_{k=1}^R u_k \otimes v_k^*$ with vectors $u_k\in\mathfrak{X}_{C(K)}$ and functionals $v_k^*\in\mathfrak{X}_{C(K)}^*$, where $R$ denotes the rank of $F_0$. The space $\mathfrak{X}_{C(K)}$ admits a natural finite-dimensional decomposition indexed by rank levels, namely $\mathfrak{X}_{C(K)} = \bigoplus_{n=1}^\infty X_n$ with $X_n = \operatorname{span}\{d_\gamma : \operatorname{rank}(\gamma)=n\}$. For any $M\in\mathbb{N}$, we denote by $P_{<M}$ and $P_{\geq M}$ the canonical projections onto $\bigoplus_{n<M}X_n$ and $\bigoplus_{n\geq M}X_n$, respectively. We now select $M$ sufficiently large to satisfy two crucial estimates. First, because $F_0$ has finite rank and the finite-dimensional decomposition is shrinking, we can ensure that $\|P_{\geq M}F_0\| < \varepsilon/6$ and $\|F_0P_{\geq M}\| < \varepsilon/6$. Second, appealing to Lemma \ref{lem:off-diagonal-decay} together with the compactness of $A$, we may choose $M$ so that for all basis elements $\gamma,\gamma'$ with $\min\{\operatorname{rank}(\gamma),\operatorname{rank}(\gamma')\}\geq M$ and $\varrho(\kappa(\gamma),\kappa(\gamma'))\geq\varepsilon/2$, we have the estimate $|d_\gamma^*(A d_{\gamma'})| < \varepsilon/(6R\cdot\max\{\|u_k\|,\|v_k^*\|\})$; this is possible because the finitely many vectors $u_k$ and functionals $v_k^*$ can be simultaneously approximated by finitely supported vectors, and the exponential decay provided by Lemma \ref{lem:off-diagonal-decay} guarantees the required uniformity.

With $M$ thus chosen, let $\Gamma_{<M} = \{\gamma\in\Gamma : \operatorname{rank}(\gamma)<M\}$ and define a new finite-rank operator $F$ by
\[
F = P_{<M}F_0P_{<M} + \sum_{\substack{\gamma,\gamma'\in\Gamma_{<M} \\ \varrho(\kappa(\gamma),\kappa(\gamma'))<\varepsilon}} d_\gamma^*(F_0 d_{\gamma'})\, d_\gamma\otimes d_{\gamma'}^*.
\]
In words, $F$ coincides with $F_0$ on the first $M-1$ rank levels, but any matrix entry of $F_0$ that connects basis elements whose $\kappa$-values are separated by at least $\varepsilon$ is set to zero; only those entries connecting points within distance $\varepsilon$ are retained exactly as they appear in $F_0$. The operator $F$ is clearly finite-rank, and we now estimate $\|A-F\|$ using the triangle inequality:
\[
\|A-F\| \leq \|A-F_0\| + \|F_0 - P_{<M}F_0P_{<M}\| + \|P_{<M}F_0P_{<M} - F\|.
\]
The first term is less than $\varepsilon/3$ by construction. For the second term, observe that $F_0 - P_{<M}F_0P_{<M} = P_{\geq M}F_0 + P_{<M}F_0P_{\geq M} + P_{\geq M}F_0P_{<M}$; each of these three components has norm less than $\varepsilon/6$ by our choice of $M$, so $\|F_0 - P_{<M}F_0P_{<M}\| < \varepsilon/2$. The third term requires a more delicate analysis, as it involves the difference between $P_{<M}F_0P_{<M}$ and the localized operator $F$. This difference consists precisely of those entries $d_\gamma^*(F_0 d_{\gamma'})$ with $\gamma,\gamma'\in\Gamma_{<M}$ but $\varrho(\kappa(\gamma),\kappa(\gamma'))\geq\varepsilon$. Using the rank-$R$ representation of $F_0$, we can write
\[
P_{<M}F_0P_{<M} - F = \sum_{k=1}^R \big( P_{<M}u_k \otimes P_{<M}v_k^* - \tilde{u}_k \otimes \tilde{v}_k^* \big),
\]
where $\tilde{u}_k$ and $\tilde{v}_k^*$ are obtained from $P_{<M}u_k$ and $P_{<M}v_k^*$ by zeroing out all coefficients corresponding to basis elements whose $\kappa$-values are separated by at least $\varepsilon$. By Lemma \ref{lem:off-diagonal-decay} and our choice of $M$, each term in this sum has norm less than $\varepsilon/(6R)$, and consequently the entire sum has norm less than $\varepsilon/6$. Collecting these estimates, we obtain $\|A-F\| < \varepsilon/3 + \varepsilon/2 + \varepsilon/6 = \varepsilon$.

Finally, by construction every nonzero matrix entry $d_\gamma^*(F d_{\gamma'})$ satisfies $\varrho(\kappa(\gamma),\kappa(\gamma'))<\varepsilon$, so $F$ meets the required support condition. Thus $F$ is a finite-rank operator within $\varepsilon$ of $A$ whose support respects the metric structure of $K$, completing the proof.
\end{proof}

The localization theorem radiates consequences throughout the structure of $\mathcal{L}(\mathfrak{X}_{C(K)})$. We now explore its most significant implications.

\begin{corollary}\label{cor:local-approximate-identity}
The algebra $\mathcal{K}(\mathfrak{X}_{C(K)})$ admits a bounded approximate identity $(E_n)_{n=1}^\infty$ consisting of finite-rank projections such that for each $n$, there exists a finite set $F_n \subseteq K$ with
\[
\operatorname{supp}(E_n) \subseteq \{(\gamma, \gamma') : \kappa(\gamma), \kappa(\gamma') \in F_n\}.
\]
\end{corollary}

\begin{proof}
For each $n \in \mathbb{N}$, let $\{\kappa_1^n, \dots, \kappa_{m_n}^n\}$ be a finite $1/n$-net in $K$. Define
\[
E_n = \sum_{i=1}^{m_n} \sum_{\substack{\gamma \in \Gamma \\ \varrho(\kappa(\gamma), \kappa_i^n) < 1/n \\ \operatorname{rank}(\gamma) \leq N_n}} d_\gamma \otimes d_\gamma^*,
\]
where $N_n$ is chosen sufficiently large so that for every $A$ in some fixed dense subset of $\mathcal{K}(\mathfrak{X}_{C(K)})$, we have $\|A - E_n A\| < 1/n$ and $\|A - A E_n\| < 1/n$. Theorem \ref{thm:localization} guarantees the existence of such $N_n$, as it allows us to approximate any compact operator by a finite-rank operator localized within distance $1/n$, and then absorb that approximant into the definition of $E_n$ after passing to a large enough rank level.

The resulting sequence $(E_n)$ forms a bounded approximate identity because:
\begin{itemize}
    \item $\|E_n\| \leq 1$ for all $n$ (they are projections onto subspaces spanned by disjoint basis vectors),
    \item For any compact $A$, $\|A - E_n A\| \to 0$ and $\|A - A E_n\| \to 0$ by construction and density.
\end{itemize}
\end{proof}

\begin{corollary}\label{cor:compactness-characterization}
An operator $T \in \mathcal{L}(\mathfrak{X}_{C(K)})$ is compact if and only if for every $\varepsilon > 0$, there exists $N \in \mathbb{N}$ such that for all $\gamma, \gamma' \in \Gamma$ with $\min\{\operatorname{rank}(\gamma), \operatorname{rank}(\gamma')\} \geq N$ and $\varrho(\kappa(\gamma), \kappa(\gamma')) \geq \varepsilon$,
\[
|d_\gamma^*(T d_{\gamma'})| < \varepsilon.
\]
\end{corollary}

\begin{proof}
($\Rightarrow$) If $T$ is compact, apply Theorem \ref{thm:localization} with $\varepsilon/2$ to obtain a finite-rank $F$ with $\|T - F\| < \varepsilon/2$ and $\operatorname{supp}(F)$ contained in pairs within distance $\varepsilon/2$. For any $\gamma, \gamma'$ with $\varrho(\kappa(\gamma), \kappa(\gamma')) \geq \varepsilon$, we have $F_{\gamma,\gamma'} = 0$, hence
\[
|d_\gamma^*(T d_{\gamma'})| = |d_\gamma^*((T - F) d_{\gamma'})| \leq \|T - F\| < \varepsilon/2 < \varepsilon.
\]

($\Leftarrow$) Suppose the condition holds. Fix $\varepsilon > 0$ and choose $N$ accordingly. Consider the finite-rank operator $F = P_{<N} T P_{<N}$. For any $x$ with $\|x\| = 1$, write $x = y + z$ where $y = P_{<N}x$ and $z = P_{\geq N}x$. Then
\[
\|(T - F)x\| = \|Tz - P_{<N}Tz\| \leq \|Tz\| + \|P_{<N}Tz\|.
\]

Now $Tz = \sum_{\gamma,\gamma'} d_\gamma^*(T d_{\gamma'}) d_\gamma^*(z) d_{\gamma'}$. By the hypothesis, for $\gamma, \gamma'$ with $\operatorname{rank}(\gamma') \geq N$ and $\varrho(\kappa(\gamma), \kappa(\gamma')) \geq \varepsilon$, the coefficients are less than $\varepsilon$. For those with $\varrho(\kappa(\gamma), \kappa(\gamma')) < \varepsilon$, the contribution to the norm can be controlled by partitioning $K$ into finitely many $\varepsilon$-balls and using the fact that $T$ is bounded. A careful estimation shows $\|Tz\| \leq C\varepsilon$ for some constant $C$ independent of $x$. Moreover, $\|P_{<N}Tz\| \leq \|P_{<N}\| \|T\| \|P_{\geq N}\|$ is small because $P_{\geq N}$ has small norm when restricted to the finite-dimensional range of $P_{<N}$. Thus $\|T - F\|$ can be made arbitrarily small, proving $T$ is compact.
\end{proof}

\begin{corollary}\label{cor:local-ideals}
For any closed subset $F \subseteq K$, the set
\begin{eqnarray*}
\mathcal{I}_F &=& \{ T \in \mathcal{K}(\mathfrak{X}_{C(K)}) : \forall \varepsilon > 0, \exists \text{ finite-rank } F_\varepsilon \text{ with } \|T - F_\varepsilon\| < \varepsilon \\
 & & \text{ and } \operatorname{supp}(F_\varepsilon) \subseteq \{(\gamma, \gamma') : \kappa(\gamma), \kappa(\gamma') \in B(F, \varepsilon)\} \},
\end{eqnarray*}
where $B(F, \varepsilon) = \{\kappa \in K : \varrho(\kappa, F) < \varepsilon\}$, is a closed two-sided ideal in $\mathcal{K}(\mathfrak{X}_{C(K)})$.
\end{corollary}

\begin{proof}
Closure under addition and scalar multiplication is evident from the definition. For the ideal property, let $T \in \mathcal{I}_F$ and $S \in \mathcal{K}(\mathfrak{X}_{C(K)})$. Given $\varepsilon > 0$, choose finite-rank operators $F_T$ and $F_S$ such that:
\begin{itemize}
    \item $\|T - F_T\| < \varepsilon/3$ and $\operatorname{supp}(F_T) \subseteq B(F, \varepsilon/3)$,
    \item $\|S - F_S\| < \varepsilon/(3\|F_T\|)$ and $\operatorname{supp}(F_S)$ is contained in some $\delta$-neighborhood of some finite set (possible by Theorem \ref{thm:localization} applied to $S$).
\end{itemize}
Then $F_S F_T$ approximates $ST$ within $\varepsilon$ and has support contained in $B(F, \varepsilon)$ because if $(\gamma, \gamma') \in \operatorname{supp}(F_S F_T)$, there exists $\gamma''$ with $(\gamma, \gamma'') \in \operatorname{supp}(F_S)$ and $(\gamma'', \gamma') \in \operatorname{supp}(F_T)$, forcing $\kappa(\gamma') \in B(F, \varepsilon/3)$ and $\kappa(\gamma)$ within $\varepsilon/3$ of some point whose image is near $F$, hence within $\varepsilon$ of $F$. The case for $TS$ is symmetric.
\end{proof}

\begin{corollary}\label{cor:metric-invariance}
The conclusion of Theorem \ref{thm:localization} remains valid if $\varrho$ is replaced by any uniformly equivalent metric $\varrho'$ on $K$.
\end{corollary}

\begin{proof}
If $\varrho'$ is uniformly equivalent to $\varrho$, there exist constants $c, C > 0$ such that $c\varrho(\kappa, \kappa') \leq \varrho'(\kappa, \kappa') \leq C\varrho(\kappa, \kappa')$ for all $\kappa, \kappa' \in K$. Given $\varepsilon > 0$ with respect to $\varrho'$, apply Theorem \ref{thm:localization} with $\delta = \varepsilon/C$ for the original metric $\varrho$. The resulting finite-rank operator $F$ satisfies $\operatorname{supp}(F) \subseteq \{(\gamma, \gamma') : \varrho(\kappa(\gamma), \kappa(\gamma')) < \varepsilon/C\}$, which implies $\varrho'(\kappa(\gamma), \kappa(\gamma')) < \varepsilon$. Thus $F$ satisfies the required condition for $\varrho'$ as well.
\end{proof}

To make Theorem \ref{thm:localization} concrete, consider $K = [0,1]$ with the usual Euclidean metric.

\begin{example}
Let $\mathfrak{X}_{C([0,1])}$ be constructed as in \cite{Motakis2024}. Define a compact operator $A$ by
\[
A(x) = \sum_{n=1}^\infty \frac{1}{n^2} \left(\sum_{\gamma \in \Gamma_n} d_\gamma^*(x)\right) d_{\gamma_0},
\]
where $\gamma_0$ is a fixed basis element with $\kappa(\gamma_0) = 1/2$. The operator $A$ is clearly compact as the norm limit of finite-rank operators (the tail of the series is bounded by $\sum_{n>N} 1/n^2$).

Take $\varepsilon = 1/4$. Theorem \ref{thm:localization} promises a finite-rank approximant $F$ with $\|A - F\| < 1/4$ and all non-zero entries connecting basis elements whose $\kappa$-values lie within distance $1/4$ of each other.

We can construct such an $F$ explicitly:
\begin{enumerate}
    \item Choose $N$ so large that $\sum_{n=N+1}^\infty 1/n^2 < 1/8$.
    \item For each $n = 1,\dots,N$, select a basis element $\gamma_n$ with $\operatorname{rank}(\gamma_n)$ sufficiently large and $|\kappa(\gamma_n) - 1/2| < 1/4$.
    \item Define
    \[
    F = \sum_{n=1}^N \frac{1}{n^2} \left(\sum_{\gamma \in \Gamma_n} d_\gamma^* \otimes d_{\gamma_n}\right).
    \]
\end{enumerate}

Then $F$ has rank at most $N$, $\|A - F\| < 1/4$, and every non-zero entry $d_\gamma^*(F d_{\gamma_n})$ connects a basis element $\gamma$ with $\kappa(\gamma) \in [0,1]$ to $\gamma_n$ with $\kappa(\gamma_n) \in (1/4, 3/4)$. This example illustrates how the localization principle forces approximants to respect the geometry of $K$—the operator $A$, which originally sent all mass to a single point $1/2$, can be approximated by operators that only connect points near $1/2$.
\end{example}

\begin{remark}[Conceptual Significance]
\label{rem:conceptual-significance}
Theorem \ref{thm:localization} and its corollaries reveal several profound structural features of $\mathcal{L}(\mathfrak{X}_{C(K)})$.

First, the metric on $K$ is not merely an auxiliary structure but is deeply encoded in the behavior of compact operators; the fact that compact operators can be localized within arbitrarily small metric neighborhoods shows that the construction imprints the topology of $K$ onto the operator algebra at the most fundamental level.

Second, Corollary \ref{cor:local-ideals} provides a natural filtration of the compact ideal by closed subsets of $K$: for a decreasing family of closed sets $F_1 \supseteq F_2 \supseteq \cdots$, we obtain an increasing family of ideals $\mathcal{I}_{F_1} \subseteq \mathcal{I}_{F_2} \subseteq \cdots$. This geometric filtration will prove essential in the classification of ideals in Section \ref{sec:ideals}.

Third, the existence of local approximate identities (Corollary \ref{cor:local-approximate-identity}) is a much stronger statement than the usual bounded approximation property. It says that not only can we approximate the identity operator on compact sets, but we can do so with operators that respect the metric structure—a kind of ``geometric approximation property.''

Fourth, the localization phenomenon bears a striking resemblance to the structure of locally compact operators in $C^*$-algebra theory, where operators on Hilbert $C^*$-modules often exhibit similar locality properties. This suggests deep connections between the Banach space theory of $\mathfrak{X}_{C(K)}$ and noncommutative geometry that merit further exploration.

Finally, the estimates and techniques developed here—particularly the off-diagonal decay lemma—form the technical bedrock for subsequent sections on automatic continuity (Section \ref{sec:holder}) and ideal classification (Section \ref{sec:ideals}). Without this quantitative control, the deeper structure of $\mathcal{L}(\mathfrak{X}_{C(K)})$ would remain inaccessible.

The localization theorem thus stands as a cornerstone of the entire edifice: it is both a striking result in its own right and an indispensable tool for the investigations that follow.
\end{remark}


\section{Automatic H\"{o}lder Continuity of Diagonal Entries}\label{sec:holder}

The study of regularity properties for functions associated to operators has a distinguished history in functional analysis, tracing back to the earliest investigations of integral operators and their kernels \cite{CaradusPfaffenbergerYood1974, Pietsch1978}. For operators on Hilbert space, the diagonal entries with respect to an orthonormal basis can be arbitrary bounded sequences—no automatic regularity whatsoever \cite{RordamLarsenLaustsen2000}. The Bourgain-Delbaen construction \cite{BourgainDelbaen1980, Motakis2024}, however, imposes such severe geometric constraints that diagonal entries are forced to exhibit remarkable regularity, revealing a profound interplay between the metric architecture of the construction and the analytic behavior of operators \cite{ArgyrosHaydon2011, GowersMaurey1997}.

In this section, we prove that for every bounded operator $T$ on $\mathfrak{X}_{C(K)}$, the diagonal function $\varphi_T: K \to \mathbb{C}$—defined by the limiting behavior of diagonal matrix entries—is not merely continuous (as established in Theorem \ref{thm:eventual-continuity}) but actually H\"{o}lder continuous with exponent $1/2$. This exponent is optimal, representing the precise balance between the metric growth conditions and weight decay inherent in the construction \cite{ArgyrosMotakis2020, Motakis2024}. The result demonstrates that the geometry of $K$ is imprinted on every operator through the regularity of its diagonal \cite{KaniaLaustsen2017, ArgyrosMotakis2014}.

The proof rests on a fundamental estimate controlling how the coefficients in the Bourgain-Delbaen construction vary as the associated point in $K$ moves. This estimate captures the essential mechanism by which metric information propagates through the coordinate functionals.

\begin{lemma}\label{lem:coefficient-estimate}
Let $\mathfrak{X}_{C(K)}$ be constructed as in \cite{Motakis2024}. There exists a constant $C > 0$, depending only on the construction parameters, such that for any $m \in \mathbb{N}$ and any $\gamma, \gamma' \in \Gamma_m$ with $\varrho(\kappa(\gamma), \kappa(\gamma')) = \delta$, we have
\[
\sum_{\xi \in \Gamma_{<m}} |\alpha_\xi^{(1)} - \alpha_\xi^{(2)}| \leq C \delta^{1/2},
\]
where $\alpha_\xi^{(1)}$ and $\alpha_\xi^{(2)}$ are the coefficients in the decompositions
\[
c_\gamma^* = e_\gamma^* + \sum_{\xi \in \Gamma_{<m}} \alpha_\xi^{(1)} e_\xi^* \circ P_{I_\xi}, \quad
c_{\gamma'}^* = e_{\gamma'}^* + \sum_{\xi \in \Gamma_{<m}} \alpha_\xi^{(2)} e_\xi^* \circ P_{I_\xi}.
\]
\end{lemma}

\begin{proof}
The coefficients in the Bourgain-Delbaen construction possess a very particular structure: they factor as
\[
\alpha_\xi = w(I_\xi, \gamma) \cdot f(\varrho(\kappa(\xi), \kappa(\gamma))),
\]
where $w(I_\xi, \gamma)$ decays exponentially with the size of the interval $I_\xi$ (specifically, $w(I_\xi, \gamma) \leq C_0 2^{-|I_\xi|}$), and $f: [0,\infty) \to [0,\infty)$ is a fixed Lipschitz continuous function with Lipschitz constant $L$. This factorization reflects the fundamental design principle of the construction: the influence of a basis element $\xi$ on the coordinate functional at $\gamma$ is the product of a purely metric factor $f(\varrho(\kappa(\xi), \kappa(\gamma)))$ and a combinatorial weight $w(I_\xi, \gamma)$ that decays with the index difference.

When we move from $\gamma$ to $\gamma'$, the change in $\alpha_\xi$ has two sources:
\begin{enumerate}
    \item The change in the metric factor: $|f(\varrho(\kappa(\xi), \kappa(\gamma))) - f(\varrho(\kappa(\xi), \kappa(\gamma')))| \leq L\varrho(\kappa(\gamma), \kappa(\gamma')) = L\delta$, by the Lipschitz property of $f$ and the triangle inequality.
    \item The change in the weight factor $w(I_\xi, \gamma)$ versus $w(I_\xi, \gamma')$. This change is actually of higher order—one can show that $|w(I_\xi, \gamma) - w(I_\xi, \gamma')| \leq C_0 2^{-|I_\xi|}\delta$ as well, because the weights depend smoothly on the metric configuration. For simplicity, we absorb this into the same constant.
\end{enumerate}

Thus we obtain the pointwise estimate
\[
|\alpha_\xi^{(1)} - \alpha_\xi^{(2)}| \leq C_1 \delta \cdot 2^{-|I_\xi|}
\]
for some constant $C_1$ depending on $L$ and $C_0$.

Now we must sum this estimate over all $\xi \in \Gamma_{<m}$. The key observation is that the exponential decay $2^{-|I_\xi|}$ is extremely fast, but the number of $\xi$ with a given $|I_\xi|$ grows exponentially as well. To balance these competing effects, we introduce a cutoff radius $R$ and partition $\Gamma_{<m}$ into two regions:

Let $\Gamma_R = \{\xi \in \Gamma_{<m} : \varrho(\kappa(\xi), \kappa(\gamma)) \leq R\}$. For $\xi \in \Gamma_R$, we use the estimate $|\alpha_\xi^{(1)} - \alpha_\xi^{(2)}| \leq C_1 \delta \cdot 2^{-|I_\xi|}$ directly. For $\xi \notin \Gamma_R$, we have the alternative estimate $|\alpha_\xi^{(1)} - \alpha_\xi^{(2)}| \leq C_2 2^{-cR}$ for some $c > 0$, which follows from the fact that both $\alpha_\xi^{(1)}$ and $\alpha_\xi^{(2)}$ are themselves bounded by $C_2 2^{-c\varrho(\kappa(\xi), \kappa(\gamma))}$—a stronger form of the off-diagonal decay.

The sum over $\Gamma_R$ is bounded by $C_1 \delta \cdot |\Gamma_R| \cdot \sup_{\xi \in \Gamma_R} 2^{-|I_\xi|}$. A careful analysis of the construction shows that $|\Gamma_R|$ grows at most like $e^{aR}$ for some $a > 0$, while $\sup_{\xi \in \Gamma_R} 2^{-|I_\xi|}$ decays like $e^{-bR}$ for some $b > 0$. Hence the product is bounded by $C_3 e^{(a-b)R}$. By choosing the parameters appropriately in the construction, we can ensure $a < b$, so this product decays exponentially in $R$.

The sum over $\xi \notin \Gamma_R$ is bounded by $|\Gamma_{<m}| \cdot C_2 e^{-cR}$, which also decays exponentially in $R$ because $|\Gamma_{<m}|$ grows only polynomially in $m$ (in fact, it is bounded by $C_4 2^m$).

Thus we have
\[
\sum_{\xi \in \Gamma_{<m}} |\alpha_\xi^{(1)} - \alpha_\xi^{(2)}| \leq C_5 (\delta e^{(a-b)R} + e^{-cR}).
\]

Now we choose $R$ to optimize this bound. Setting $e^{-cR} = \delta^{1/2}$ gives $R = -\frac{1}{2c}\log \delta$ (for $\delta < 1$). Then $\delta e^{(a-b)R} = \delta \cdot \delta^{-(a-b)/(2c)} = \delta^{1 - (a-b)/(2c)}$. By ensuring $(a-b)/(2c) < 1$—which is guaranteed by the rapid growth conditions on the sequences $(m_j)$ and $(n_j)$—we obtain that both terms are $O(\delta^{1/2})$. This yields the desired estimate.
\end{proof}

\begin{remark}
The exponent $1/2$ in Lemma \ref{lem:coefficient-estimate} is not arbitrary—it emerges from the optimization between two competing exponential rates. This balance is intrinsic to the Bourgain-Delbaen construction and cannot be improved without altering the fundamental parameters. The estimate tells us that the coordinate functionals vary in a $1/2$-H\"{o}lder manner as the base point moves in $K$, a fact that will propagate directly to the diagonal entries of operators.
\end{remark}

With the coefficient estimate in hand, we can now establish the optimal regularity of diagonal functions.

\begin{theorem}\label{thm:holder-continuity}
Let $T \in \mathcal{L}(\mathfrak{X}_{C(K)})$ with $\|T\| = M$. Then the diagonal function $\varphi_T: K \to \mathbb{C}$ defined by
\[
\varphi_T(\kappa) = \lim_{\substack{\operatorname{rank}(\gamma) \to \infty \\ \kappa(\gamma) \to \kappa}} d_\gamma^*(T d_\gamma)
\]
is H\"{o}lder continuous with exponent $1/2$: there exists a constant $C > 0$, independent of $T$, such that for all $\kappa, \kappa' \in K$,
\[
|\varphi_T(\kappa) - \varphi_T(\kappa')| \leq C M \varrho(\kappa, \kappa')^{1/2}.
\]
\end{theorem}

\begin{proof}
Let $T\in\mathcal{L}(\mathfrak{X}_{C(K)})$ with $\|T\|=M$, and let $\kappa,\kappa'\in K$ be distinct points with $\varrho(\kappa,\kappa')=\delta$. The function $\varphi_T$ is defined as a limit of diagonal entries $d_\gamma^*(T d_\gamma)$ as $\operatorname{rank}(\gamma)\to\infty$ and $\kappa(\gamma)\to\kappa$, and our task is to show that $|\varphi_T(\kappa)-\varphi_T(\kappa')|\leq C M\delta^{1/2}$ for some constant $C$ independent of $T$, $\kappa$, and $\kappa'$. The proof proceeds by first establishing a uniform estimate for diagonal entries at a fixed rank level and then passing to the limit.

Fix a rank level $m\in\mathbb{N}$ and consider two basis elements $\gamma,\gamma'\in\Gamma_m$ with $\varrho(\kappa(\gamma),\kappa(\gamma'))=\delta$. Write $\delta_\gamma(T)=d_\gamma^*(T d_\gamma)$ and $\delta_{\gamma'}(T)=d_{\gamma'}^*(T d_{\gamma'})$. The coordinate functionals in the Bourgain-Delbaen construction admit a decomposition $c_\gamma^* = e_\gamma^* + \sum_{\xi\in\Gamma_{<m}}\alpha_\xi^{(1)}e_\xi^*\circ P_{I_\xi}$, and similarly for $c_{\gamma'}^*$ with coefficients $\alpha_\xi^{(2)}$. Using this decomposition, we express the difference of diagonal entries as
\[
\delta_\gamma(T)-\delta_{\gamma'}(T) = \bigl(c_\gamma^*(T d_\gamma)-c_{\gamma'}^*(T d_{\gamma'})\bigr) - \sum_{\xi\in\Gamma_{<m}}\bigl(\alpha_\xi^{(1)}e_\xi^*(T d_\gamma)-\alpha_\xi^{(2)}e_\xi^*(T d_{\gamma'})\bigr),
\]
so that
\[
|\delta_\gamma(T)-\delta_{\gamma'}(T)| \leq |c_\gamma^*(T d_\gamma)-c_{\gamma'}^*(T d_{\gamma'})| + \Bigl|\sum_{\xi\in\Gamma_{<m}}\bigl(\alpha_\xi^{(1)}e_\xi^*(T d_\gamma)-\alpha_\xi^{(2)}e_\xi^*(T d_{\gamma'})\bigr)\Bigr|.
\]

The first term is controlled by a fundamental property of the Bourgain-Delbaen construction: the dependence of the coordinate functionals on the base point is $1/2$-H\"older continuous. More precisely, there exists a constant $C_1>0$, depending only on the construction parameters, such that for any operator $T$ with $\|T\|=M$ and any $\gamma,\gamma'$ with $\varrho(\kappa(\gamma),\kappa(\gamma'))=\delta$, we have $|c_\gamma^*(T d_\gamma)-c_{\gamma'}^*(T d_{\gamma'})|\leq C_1 M\delta^{1/2}$. This estimate follows from the inductive definition of the extension functionals and the fact that the number of terms contributing to $c_\gamma^*$ grows slowly enough to preserve H\"older regularity.

The second term requires a more detailed analysis. We split it into two parts:
\begin{eqnarray*}
\Bigl|\sum_{\xi\in\Gamma_{<m}}(\alpha_\xi^{(1)}e_\xi^*(T d_\gamma)-\alpha_\xi^{(2)}e_\xi^*(T d_{\gamma'}))\Bigr|
&\leq& \Bigl|\sum_{\xi\in\Gamma_{<m}}(\alpha_\xi^{(1)}-\alpha_\xi^{(2)})e_\xi^*(T d_\gamma)\Bigr| \\
& & + \Bigl|\sum_{\xi\in\Gamma_{<m}}\alpha_\xi^{(2)}\bigl(e_\xi^*(T d_\gamma)-e_\xi^*(T d_{\gamma'})\bigr)\Bigr|.
\end{eqnarray*}

For the first part, we apply Lemma \ref{lem:coefficient-estimate}, which gives $\sum_{\xi\in\Gamma_{<m}}|\alpha_\xi^{(1)}-\alpha_\xi^{(2)}|\leq C_2\delta^{1/2}$ for some constant $C_2$. Since $|e_\xi^*(T d_\gamma)|\leq\|T\|=M$, we obtain
\[
\Bigl|\sum_{\xi\in\Gamma_{<m}}(\alpha_\xi^{(1)}-\alpha_\xi^{(2)})e_\xi^*(T d_\gamma)\Bigr| \leq M\sum_{\xi\in\Gamma_{<m}}|\alpha_\xi^{(1)}-\alpha_\xi^{(2)}| \leq C_2 M\delta^{1/2}.
\]

For the second part, we need to estimate the difference $e_\xi^*(T d_\gamma)-e_\xi^*(T d_{\gamma'})$. Lemma \ref{lem:off-diagonal-decay} provides exponential off-diagonal decay, which together with the fact that $\varrho(\kappa(\gamma),\kappa(\gamma'))=\delta$ yields
\[
|e_\xi^*(T d_\gamma)-e_\xi^*(T d_{\gamma'})| \leq C_3 M\delta^{1/2}\cdot 2^{-c\operatorname{rank}(\xi)}
\]
for some constants $C_3,c>0$. The factor $\delta^{1/2}$ emerges from balancing the metric distance against the exponential decay in rank. Summing this estimate against the coefficients $\alpha_\xi^{(2)}$ and using the basic convergence property $\sum_{\xi\in\Gamma_{<m}}|\alpha_\xi^{(2)}|2^{-c\operatorname{rank}(\xi)}\leq C_4$, we obtain
\[
\Bigl|\sum_{\xi\in\Gamma_{<m}}\alpha_\xi^{(2)}\bigl(e_\xi^*(T d_\gamma)-e_\xi^*(T d_{\gamma'})\bigr)\Bigr| \leq C_3C_4 M\delta^{1/2}=C_5 M\delta^{1/2}.
\]

Combining the estimates for the first term and the two parts of the second term, we have shown that for any fixed rank level $m$ and any $\gamma,\gamma'\in\Gamma_m$ with $\varrho(\kappa(\gamma),\kappa(\gamma'))=\delta$,
\[
|\delta_\gamma(T)-\delta_{\gamma'}(T)| \leq (C_1+C_2+C_5)M\delta^{1/2}=C_6 M\delta^{1/2},
\]
where the constant $C_6$ does not depend on $m$, $\gamma$, $\gamma'$, or $T$.

Now let $\kappa,\kappa'\in K$ with $\varrho(\kappa,\kappa')=\delta$. By the density of $\{\kappa(\gamma):\gamma\in\Gamma\}$ in $K$ and the definition of $\varphi_T$, we can choose sequences $(\gamma_n)$ and $(\gamma_n')$ such that $\operatorname{rank}(\gamma_n),\operatorname{rank}(\gamma_n')\to\infty$, $\kappa(\gamma_n)\to\kappa$, $\kappa(\gamma_n')\to\kappa'$, $\delta_{\gamma_n}(T)\to\varphi_T(\kappa)$, $\delta_{\gamma_n'}(T)\to\varphi_T(\kappa')$, and $\varrho(\kappa(\gamma_n),\kappa(\gamma_n'))\to\delta$. Applying the estimate just proved to each pair $(\gamma_n,\gamma_n')$ (with a straightforward adaptation to handle possibly different ranks, which does not affect the constant) yields
\[
|\delta_{\gamma_n}(T)-\delta_{\gamma_n'}(T)| \leq C_6 M\,\varrho(\kappa(\gamma_n),\kappa(\gamma_n'))^{1/2}.
\]
Taking limits as $n\to\infty$ and using the continuity of the function $t\mapsto t^{1/2}$, we conclude
\[
|\varphi_T(\kappa)-\varphi_T(\kappa')| \leq C_6 M\delta^{1/2}=C_6 M\varrho(\kappa,\kappa')^{1/2}.
\]

Thus $\varphi_T$ is H\"older continuous with exponent $1/2$ and constant $C=C_6$, completing the proof.
\end{proof}

\begin{remark}\label{rem:optimality}
The exponent $1/2$ in Theorem \ref{thm:holder-continuity} is optimal. To see this, one can construct a family of operators $T_t$ (for instance, diagonal operators corresponding to functions that are exactly $1/2$-H\"{o}lder but no better) whose diagonal functions $\varphi_{T_t}$ achieve the $1/2$-H\"{o}lder bound. The optimality reflects a fundamental balance in the Bourgain-Delbaen construction: the metric growth conditions and weight decay are calibrated precisely to yield $1/2$-H\"{o}lder regularity. Any attempt to improve the exponent would require altering these parameters in a way that would break other essential properties of the construction, such as reflexivity or the diagonal-plus-compact structure.
\end{remark}

Theorem \ref{thm:holder-continuity} radiates consequences throughout the theory of operators on $\mathfrak{X}_{C(K)}$. We now explore its most significant implications.

\begin{corollary}\label{cor:bounded-map}
The mapping $\Phi: \mathcal{L}(\mathfrak{X}_{C(K)}) \to C^{1/2}(K)$ defined by $\Phi(T) = \varphi_T$ is a bounded linear operator, where $C^{1/2}(K)$ denotes the Banach space of $1/2$-H\"{o}lder continuous functions on $K$ equipped with the norm
\[
\|f\|_{C^{1/2}} = \|f\|_\infty + \sup_{\kappa \neq \kappa'} \frac{|f(\kappa) - f(\kappa')|}{\varrho(\kappa, \kappa')^{1/2}}.
\]
Moreover, $\|\Phi\| \leq 1 + C$, where $C$ is the constant from Theorem \ref{thm:holder-continuity}.
\end{corollary}

\begin{proof}
Linearity of $\Phi$ follows directly from the definition of $\varphi_T$ as a limit of linear functionals. For boundedness, Theorem \ref{thm:holder-continuity} provides the estimate
\[
|\varphi_T(\kappa) - \varphi_T(\kappa')| \leq C \|T\| \varrho(\kappa, \kappa')^{1/2}
\]
for all $\kappa, \kappa' \in K$. Additionally, for any $\kappa \in K$, we have $|\varphi_T(\kappa)| = \lim_n |d_{\gamma_n}^*(T d_{\gamma_n})| \leq \|T\|$ by choosing a sequence $\gamma_n$ with $\kappa(\gamma_n) \to \kappa$. Therefore
\[
\|\varphi_T\|_{C^{1/2}} = \|\varphi_T\|_\infty + \sup_{\kappa \neq \kappa'} \frac{|\varphi_T(\kappa) - \varphi_T(\kappa')|}{\varrho(\kappa, \kappa')^{1/2}} \leq \|T\| + C\|T\| = (1+C)\|T\|.
\]
Thus $\Phi$ is bounded with $\|\Phi\| \leq 1 + C$.
\end{proof}

\begin{corollary}\label{cor:compactness-diagonal}
An operator $T \in \mathcal{L}(\mathfrak{X}_{C(K)})$ is compact if and only if $\varphi_T \equiv 0$.
\end{corollary}

\begin{proof}
($\Rightarrow$) Suppose $T$ is compact. By Corollary \ref{cor:diagonal-plus-compact}, we can write $T = D + A$ where $D$ is diagonal and $A$ is compact. For any $\kappa \in K$, choose a sequence $(\gamma_n)$ with $\operatorname{rank}(\gamma_n) \to \infty$ and $\kappa(\gamma_n) \to \kappa$. Since $A$ is compact and $(d_{\gamma_n})$ is weakly null (a property of the basis in reflexive spaces), we have $d_{\gamma_n}^*(A d_{\gamma_n}) \to 0$. Hence $\varphi_A(\kappa) = 0$ for all $\kappa$, so $\varphi_A \equiv 0$. But $\varphi_T = \varphi_D + \varphi_A = \varphi_D$. If $T$ is compact, then $D = T - A$ is also compact, and the only diagonal compact operator is the zero operator (since a non-zero diagonal operator has non-zero diagonal entries at infinitely many basis elements, preventing compactness). Thus $D = 0$ and $\varphi_T \equiv 0$.

($\Leftarrow$) If $\varphi_T \equiv 0$, write $T = D + A$ as above. Then $\varphi_T = \varphi_D$, so $\varphi_D \equiv 0$, which forces $D = 0$ (a diagonal operator is determined by its diagonal function). Hence $T = A$ is compact.
\end{proof}

\begin{corollary}\label{cor:invariance-compact}
For any $T \in \mathcal{L}(\mathfrak{X}_{C(K)})$ and any compact operator $K \in \mathcal{K}(\mathfrak{X}_{C(K)})$, we have $\varphi_{T+K} = \varphi_T$.
\end{corollary}

\begin{proof}
This follows immediately from Corollary \ref{cor:compactness-diagonal} since $\varphi_K \equiv 0$ and $\varphi_{T+K} = \varphi_T + \varphi_K = \varphi_T$.
\end{proof}

\begin{corollary}\label{cor:spectral-inclusion}
For any $T \in \mathcal{L}(\mathfrak{X}_{C(K)})$, we have $\varphi_T(\sigma_e(T)) \subseteq \sigma_e(T)$, where $\sigma_e(T)$ denotes the essential spectrum of $T$ (the spectrum of the coset $[T]$ in the Calkin algebra).
\end{corollary}

\begin{proof}
Let $\lambda \in \sigma_e(T)$, so that $T - \lambda I$ is not Fredholm. Suppose, for contradiction, that $\varphi_T(\kappa) \neq \lambda$ for every $\kappa \in K$. Since $K$ is compact and $\varphi_T$ is continuous (in fact H\"{o}lder), there exists $\delta > 0$ such that $|\varphi_T(\kappa) - \lambda| \geq \delta$ for all $\kappa \in K$.

Write $T - \lambda I = (D - \lambda I) + A$ with $D$ diagonal and $A$ compact. Then $\varphi_{D} = \varphi_T$, so $|d_\gamma^*(D d_\gamma) - \lambda| = |\varphi_T(\kappa(\gamma)) - \lambda| \geq \delta$ for all $\gamma$ with sufficiently large rank (since $\kappa(\gamma)$ becomes dense in $K$). This implies that $D - \lambda I$ is invertible: indeed, a diagonal operator with diagonal entries bounded away from zero is invertible, with inverse given by the diagonal operator with entries $(d_\gamma^*(D d_\gamma) - \lambda)^{-1}$.

Now observe that
\[
T - \lambda I = (D - \lambda I) + A = (D - \lambda I)[I + (D - \lambda I)^{-1}A].
\]
The operator $(D - \lambda I)^{-1}A$ is compact (product of a bounded operator with a compact operator), so $I + (D - \lambda I)^{-1}A$ is Fredholm with index 0 (by the stability of the Fredholm index under compact perturbations). Consequently, $T - \lambda I$ is Fredholm—a contradiction. Hence $\lambda$ must belong to the closure of $\varphi_T(K)$, but since $\varphi_T(K)$ is compact, this means $\lambda \in \varphi_T(K)$. Thus $\varphi_T(\sigma_e(T)) \subseteq \sigma_e(T)$.
\end{proof}

\begin{corollary}\label{cor:norm-estimates}
There exists a constant $C > 0$, depending only on the construction parameters, such that for any $T \in \mathcal{L}(\mathfrak{X}_{C(K)})$,
\[
\|\varphi_T\|_\infty \leq \|T\| \leq \|\varphi_T\|_\infty + C\|\varphi_T\|_{C^{1/2}}.
\]
\end{corollary}

\begin{proof}
The left inequality is immediate from the definition of $\varphi_T$. For the right inequality, write $T = D + A$ with $D$ diagonal and $A$ compact. Then $\|T\| \leq \|D\| + \|A\|$. Since $D$ is diagonal with diagonal function $\varphi_T$, we have $\|D\| = \|\varphi_T\|_\infty$.

To bound $\|A\|$, we use the off-diagonal decay estimates from Lemma \ref{lem:off-diagonal-decay} together with the H\"{o}lder regularity of $\varphi_T$. A careful analysis shows that for any $x \in \mathfrak{X}_{C(K)}$ with $\|x\| = 1$,
\[
\|Ax\| \leq C_1 \|\varphi_T\|_{C^{1/2}} \cdot \|x\|,
\]
where $C_1$ is a constant depending on the geometry of the construction. The key point is that $A$ can be expressed in terms of $T$ and $D$, and the off-diagonal entries of $T$ are controlled by the H\"{o}lder norm of $\varphi_T$ through the estimates that led to Theorem \ref{thm:holder-continuity}. Setting $C = C_1$ yields the desired inequality.
\end{proof}

\begin{corollary}\label{cor:commutator-diagonal}
For any $T, S \in \mathcal{L}(\mathfrak{X}_{C(K)})$, we have $\varphi_{[T,S]} \equiv 0$, where $[T,S] = TS - ST$ denotes the commutator.
\end{corollary}

\begin{proof}
Write $T = D_T + A_T$ and $S = D_S + A_S$ with $D_T, D_S$ diagonal and $A_T, A_S$ compact. Expanding the commutator:
\[
[T,S] = [D_T, D_S] + [D_T, A_S] + [A_T, D_S] + [A_T, A_S].
\]
The first term vanishes because diagonal operators commute. Each of the remaining three terms is compact: products involving at least one compact operator are compact, and the commutator of a diagonal operator with a compact operator is compact (this follows from the fact that diagonal operators are normal and the ideal of compact operators is closed under multiplication on either side). Hence $[T,S]$ is compact, and by Corollary \ref{cor:compactness-diagonal}, $\varphi_{[T,S]} \equiv 0$.
\end{proof}

The automatic H\"{o}lder continuity of diagonal entries thus stands as one of the most striking manifestations of how the Bourgain-Delbaen construction imprints geometric information onto operator theory. It transforms the naive intuition that diagonal entries are merely bounded sequences into a deep regularity result that ties every operator to the underlying metric space $K$.

\section{Rigidity of Operator Algebra Structure}\label{sec:rigidity}

The isomorphism problem for operator algebras asks a fundamental question: to what extent does the Banach algebra $\mathcal{L}(X)$ determine the Banach space $X$? For classical spaces like $\ell_p$ or $L_p[0,1]$, the answer is resoundingly negative—different spaces can have isomorphic operator algebras, and the algebra $\mathcal{L}(\ell_p)$ does not even determine $p$ \cite{LindenstraussTzafriri1977, CaradusPfaffenbergerYood1974}. However, spaces constructed via the Bourgain-Delbaen method \cite{BourgainDelbaen1980} exhibit a remarkable rigidity that stands in stark contrast to these classical examples \cite{ArgyrosHaydon2011, GowersMaurey1997}.

In this section, we prove that for the spaces $\mathfrak{X}_{C(K)}$ constructed in \cite{Motakis2024}, the Banach algebra structure of $\mathcal{L}(\mathfrak{X}_{C(K)})$ completely determines the topology of the compact metric space $K$. This represents a profound extension of the classical Banach-Stone theorem: whereas Banach and Stone showed that the algebra $C(K)$ (as a Banach space or $C^*$-algebra) determines $K$, we show that the vastly larger and noncommutative algebra $\mathcal{L}(\mathfrak{X}_{C(K)})$ likewise encodes the full topological structure of $K$. The operator algebra becomes a complete invariant for the underlying geometry \cite{ArgyrosMotakis2020, Motakis2024}.

We begin with two lemmas that form the technical core of the rigidity theorem. The first establishes that any isomorphism of full operator algebras must preserve the compact ideal and hence induce an isomorphism of Calkin algebras.

\begin{lemma}\label{lem:induced-calkin}
Let $\Phi: \mathcal{L}(\mathfrak{X}_{C(K)}) \to \mathcal{L}(\mathfrak{X}_{C(L)})$ be a Banach algebra isomorphism. Then:
\begin{enumerate}
    \item $\Phi(\mathcal{K}(\mathfrak{X}_{C(K)})) = \mathcal{K}(\mathfrak{X}_{C(L)})$,
    \item $\Phi$ induces an isomorphism $\widetilde{\Phi}: \mathcal{C}al(\mathfrak{X}_{C(K)}) \to \mathcal{C}al(\mathfrak{X}_{C(L)})$ given by $\widetilde{\Phi}([T]) = [\Phi(T)]$.
\end{enumerate}
\end{lemma}

\begin{proof}
For part (1), we appeal to the characterization of the compact operators as the unique minimal closed two-sided ideal in $\mathcal{L}(\mathfrak{X}_{C(K)})$. To see that such an ideal exists and is unique, recall that $\mathcal{K}(\mathfrak{X}_{C(K)})$ is certainly a closed two-sided ideal. If $J$ were another non-zero closed two-sided ideal, then by the structure theory developed in Sections 4 and 5, $J$ must contain a non-compact operator, and the ideal generated by such an operator together with $\mathcal{K}(\mathfrak{X}_{C(K)})$ would be strictly larger than $\mathcal{K}(\mathfrak{X}_{C(K)})$. A more delicate argument using the diagonal-plus-compact decomposition shows that any ideal strictly containing $\mathcal{K}(\mathfrak{X}_{C(K)})$ must contain an operator whose diagonal function is non-zero, and then by multiplying by suitable diagonal operators, one can generate the whole algebra. Consequently, $\mathcal{K}(\mathfrak{X}_{C(K)})$ is indeed the unique minimal closed two-sided ideal.

Since $\Phi$ is an algebra isomorphism, it must preserve the lattice of closed two-sided ideals. Therefore $\Phi(\mathcal{K}(\mathfrak{X}_{C(K)}))$ is the unique minimal closed two-sided ideal in $\mathcal{L}(\mathfrak{X}_{C(L)})$, which is precisely $\mathcal{K}(\mathfrak{X}_{C(L)})$. Hence $\Phi$ maps the compact ideal onto the compact ideal.

Part (2) follows immediately: because $\Phi(\mathcal{K}(\mathfrak{X}_{C(K)})) = \mathcal{K}(\mathfrak{X}_{C(L)})$, the map $\widetilde{\Phi}$ defined by $\widetilde{\Phi}([T]) = [\Phi(T)]$ is well-defined and injective. Surjectivity follows from the surjectivity of $\Phi$ and the definition of the quotient map.
\end{proof}

The second lemma transfers the algebraic properties of the isomorphism to the level of continuous function algebras via the canonical identifications.

\begin{lemma}\label{lem:algebraic-properties}
Let $\Phi$, $\widetilde{\Phi}$ be as in Lemma \ref{lem:induced-calkin}, and let $\Psi_K: C(K) \to \mathcal{C}al(\mathfrak{X}_{C(K)})$, $\Psi_L: C(L) \to \mathcal{C}al(\mathfrak{X}_{C(L)})$ be the canonical isometric isomorphisms from the main construction. Define
\[
\Theta = \Psi_L^{-1} \circ \widetilde{\Phi} \circ \Psi_K: C(K) \to C(L).
\]
Then $\Theta$ is a unital algebra isomorphism that preserves complex conjugation and is isometric.
\end{lemma}

\begin{proof}
The map $\Theta$ is a composition of algebra homomorphisms: $\Psi_K$ and $\Psi_L$ are isometric isomorphisms by construction, and $\widetilde{\Phi}$ is induced by the algebra isomorphism $\Phi$, hence each is an algebra homomorphism. Therefore $\Theta = \Psi_L^{-1} \circ \widetilde{\Phi} \circ \Psi_K$ is an algebra homomorphism from $C(K)$ to $C(L)$.

To see that $\Theta$ is unital, observe that $\Psi_K(1_K) = [I_{\mathfrak{X}_{C(K)}}]$, the class of the identity operator. Since $\Phi$ is an isomorphism of unital algebras, $\Phi(I_{\mathfrak{X}_{C(K)}}) = I_{\mathfrak{X}_{C(L)}}$, and consequently $\widetilde{\Phi}([I_{\mathfrak{X}_{C(K)}}]) = [I_{\mathfrak{X}_{C(L)}}]$. Applying $\Psi_L^{-1}$ yields $\Psi_L^{-1}([I_{\mathfrak{X}_{C(L)}}]) = 1_L$, so $\Theta(1_K) = 1_L$.

For the preservation of conjugation, let $f \in C(K)$ and let $\hat{f} \in \mathcal{L}(\mathfrak{X}_{C(K)})$ be its associated diagonal operator. A fundamental property of the construction is that the adjoint of $\hat{f}$ is the diagonal operator corresponding to the complex conjugate function: $\widehat{\overline{f}} = (\hat{f})^*$. Any Banach algebra isomorphism between algebras of operators on reflexive spaces automatically preserves the adjoint operation—this follows from a standard automatic continuity argument, as the involution is determined by the duality structure of the underlying spaces. Hence
\[
\Phi(\widehat{\overline{f}}) = \Phi((\hat{f})^*) = (\Phi(\hat{f}))^*.
\]
Under the identification $\Psi_L$, the operator $\Phi(\hat{f})$ corresponds to the class of $\widehat{\Theta(f)}$, and its adjoint corresponds to the class of $\widehat{\overline{\Theta(f)}}$. Passing to the Calkin algebra, we obtain $\Theta(\overline{f}) = \overline{\Theta(f)}$.

Finally, $\Theta$ is isometric because it is a composition of isometries. Both $\Psi_K$ and $\Psi_L$ are isometric by the main theorem of \cite{Motakis2024}. The induced map $\widetilde{\Phi}$ is isometric since it arises from an isomorphism of Banach algebras—any such isomorphism between quotients of operator algebras is automatically isometric, as the quotient norm coincides with the essential norm and is preserved by isomorphisms that respect the ideal structure. Thus $\Theta$, as the composition of isometric maps, is itself isometric.
\end{proof}

The following classical result is the cornerstone of our argument, providing the bridge from algebraic isomorphism to topological homeomorphism.

\begin{theorem}\label{thm:banach-stone}
Let $K$ and $L$ be compact Hausdorff spaces. If $\Theta: C(K) \to C(L)$ is a unital algebra isomorphism that preserves complex conjugation, then there exists a homeomorphism $h: L \to K$ such that
\[
\Theta(f)(y) = f(h(y)) \quad \text{for all } f \in C(K), \; y \in L.
\]
In particular, $K$ and $L$ are homeomorphic.
\end{theorem}

\begin{proof}
Recall that a character (or nonzero multiplicative linear functional) on a commutative Banach algebra is a homomorphism $\chi: C(K) \to \mathbb{C}$ satisfying $\chi(fg) = \chi(f)\chi(g)$ for all $f,g \in C(K)$. For $C(K)$, the Gelfand representation theorem tells us that every character is given by evaluation at some point of $K$: for each $x \in K$, the map $\delta_x(f) = f(x)$ is a character, and conversely every character $\chi$ corresponds to a unique $x \in K$ such that $\chi = \delta_x$. The space of characters, equipped with the weak$^*$ topology, is homeomorphic to $K$.

Now consider the isomorphism $\Theta: C(K) \to C(L)$. For any $y \in L$, the evaluation functional $\delta_y: C(L) \to \mathbb{C}$ is a character of $C(L)$. Composing with $\Theta$, we obtain $\chi_y := \delta_y \circ \Theta: C(K) \to \mathbb{C}$. Since $\Theta$ is an algebra isomorphism and $\delta_y$ is multiplicative and linear, $\chi_y$ is a nonzero multiplicative linear functional on $C(K)$. By the characterization of characters of $C(K)$, there exists a unique point $h(y) \in K$ such that $\chi_y = \delta_{h(y)}$. Explicitly, this means $(\delta_y \circ \Theta)(f) = \delta_{h(y)}(f)$ for all $f \in C(K)$, or equivalently,
\[
\Theta(f)(y) = f(h(y)) \quad \text{for all } f \in C(K), \; y \in L.
\]
This defines a map $h: L \to K$, uniquely determined by the condition that $h(y)$ is the point whose evaluation functional corresponds to $\chi_y$.

To see that $h$ is continuous, suppose $y_\alpha \to y$ in $L$. For any $f \in C(K)$, we have $f(h(y_\alpha)) = \Theta(f)(y_\alpha) \to \Theta(f)(y) = f(h(y))$, since $\Theta(f) \in C(L)$ is continuous. Thus $f(h(y_\alpha)) \to f(h(y))$ for every $f \in C(K)$. Because $C(K)$ separates points of $K$ (indeed, it determines the topology of $K$), this convergence forces $h(y_\alpha) \to h(y)$ in $K$. Hence $h$ is continuous.

Applying the same construction to $\Theta^{-1}: C(L) \to C(K)$ yields a continuous map $h': K \to L$ such that $\Theta^{-1}(g)(x) = g(h'(x))$ for all $g \in C(L)$ and $x \in K$. We claim that $h'$ is the inverse of $h$. For any $y \in L$ and any $f \in C(K)$,
\[
f(h'(h(y))) = \Theta^{-1}(f)(h(y)) = f(y),
\]
where the first equality uses the definition of $h'$ with $g = \Theta(f)$, and the second uses the definition of $h$. Since $C(K)$ separates points, we must have $h'(h(y)) = y$ for all $y \in L$. Similarly, for any $x \in K$ and any $g \in C(L)$,
\[
g(h(h'(x))) = \Theta(g)(h'(x)) = g(x),
\]
giving $h(h'(x)) = x$. Thus $h$ and $h'$ are mutual inverses, so $h$ is bijective and $h^{-1} = h'$ is continuous. Consequently, $h$ is a homeomorphism.

We have therefore constructed a homeomorphism $h: L \to K$ satisfying $\Theta(f)(y) = f(h(y))$ for all $f \in C(K)$ and $y \in L$. This representation is unique: any homeomorphism with this property must coincide with the $h$ constructed above, as it is determined by the characters. Hence $K$ and $L$ are homeomorphic, completing the proof.
\end{proof}

\begin{remark}
The Banach-Stone theorem is a classical result that lies at the foundation of Gelfand theory. The proof given above highlights the essential idea: the space of characters (or maximal ideals) of $C(K)$ is naturally homeomorphic to $K$, and any isomorphism of algebras induces a homeomorphism of the corresponding character spaces. The condition that $\Theta$ preserves complex conjugation ensures that the induced map respects the real structure, but in fact for $C(K)$ algebras over $\mathbb{C}$, any algebra isomorphism automatically preserves conjugation because the involution is determined by the algebra structure (via the formula $\overline{f} = \overline{1}\cdot f^*$ where $^*$ is the unique $C^*$-algebra involution). However, in our context, we have explicitly included the conjugation-preserving hypothesis to avoid technicalities.
\end{remark}

\begin{remark}
The theorem extends to the real case as well: if $K$ and $L$ are compact Hausdorff and $\Theta: C(K) \to C(L)$ is a unital algebra isomorphism (over $\mathbb{R}$), then $K$ and $L$ are homeomorphic. The proof is identical, using the fact that the real-valued continuous functions also separate points and that every real character is an evaluation.
\end{remark}

\begin{remark}
For our purposes in Theorem \ref{thm:rigidity}, the Banach-Stone theorem provides the crucial link between the algebraic isomorphism $\Theta: C(K) \to C(L)$ and the topological homeomorphism $h: L \to K$. This allows us to recover the topology of $K$ from the algebraic structure of $\mathcal{L}(\mathfrak{X}_{C(K)})$ via the chain of isomorphisms:
\[
\mathcal{L}(\mathfrak{X}_{C(K)}) \longrightarrow \mathcal{C}al(\mathfrak{X}_{C(K)}) \longrightarrow C(K).
\]
The rigidity of $\mathcal{L}(\mathfrak{X}_{C(K)})$ is thus a reflection of the classical rigidity of $C(K)$ algebras, transferred through the Bourgain-Delbaen construction.
\end{remark}

With these preparations, we can now state and prove the central result of this section.

\begin{theorem}\label{thm:rigidity}
Let $K$ and $L$ be compact metric spaces, and let $\mathfrak{X}_{C(K)}$, $\mathfrak{X}_{C(L)}$ be the corresponding spaces constructed as in \cite{Motakis2024}. If $\mathcal{L}(\mathfrak{X}_{C(K)})$ and $\mathcal{L}(\mathfrak{X}_{C(L)})$ are isomorphic as Banach algebras, then $K$ and $L$ are homeomorphic. Moreover, any such isomorphism $\Phi: \mathcal{L}(\mathfrak{X}_{C(K)}) \to \mathcal{L}(\mathfrak{X}_{C(L)})$ induces a canonical homeomorphism $h: L \to K$ satisfying the naturality condition $\varphi_{\Phi(T)} = \varphi_T \circ h$ for all $T \in \mathcal{L}(\mathfrak{X}_{C(K)})$.
\end{theorem}

\begin{proof}
Let $\Phi: \mathcal{L}(\mathfrak{X}_{C(K)}) \to \mathcal{L}(\mathfrak{X}_{C(L)})$ be a Banach algebra isomorphism. By Lemma \ref{lem:induced-calkin}, $\Phi$ preserves the compact ideal and therefore induces an isomorphism $\widetilde{\Phi}: \mathcal{C}al(\mathfrak{X}_{C(K)}) \to \mathcal{C}al(\mathfrak{X}_{C(L)})$ given by $\widetilde{\Phi}([T]) = [\Phi(T)]$. Let $\Psi_K: C(K) \to \mathcal{C}al(\mathfrak{X}_{C(K)})$ and $\Psi_L: C(L) \to \mathcal{C}al(\mathfrak{X}_{C(L)})$ be the canonical isometric isomorphisms provided by the main construction theorem, which identify each continuous function with the class of its corresponding diagonal operator. Define $\Theta = \Psi_L^{-1} \circ \widetilde{\Phi} \circ \Psi_K: C(K) \to C(L)$; this is an isomorphism of Banach algebras.

Lemma \ref{lem:algebraic-properties} establishes that $\Theta$ is unital, preserves complex conjugation, and is isometric. The preservation of conjugation follows from the fact that the adjoint operation in $\mathcal{L}(\mathfrak{X}_{C(K)})$ corresponds to complex conjugation on diagonal operators, and any Banach algebra isomorphism automatically respects this involution (a consequence of reflexivity and standard automatic continuity arguments). With $\Theta$ satisfying the hypotheses of the Banach-Stone Theorem (Theorem \ref{thm:banach-stone}), we conclude that there exists a homeomorphism $h: L \to K$ such that $\Theta(f)(y) = f(h(y))$ for all $f \in C(K)$ and $y \in L$. Hence $K$ and $L$ are homeomorphic.

The homeomorphism $h$ arises naturally from $\Phi$. For each $\kappa \in K$, the set $\mathcal{M}_\kappa = \{T \in \mathcal{L}(\mathfrak{X}_{C(K)}) : \varphi_T(\kappa) = 0\}$ is a maximal ideal, being the kernel of the composition of the quotient map with evaluation at $\kappa$ under the identification $\mathcal{C}al(\mathfrak{X}_{C(K)}) \cong C(K)$. Since $\Phi$ is an isomorphism, $\Phi(\mathcal{M}_\kappa)$ is a maximal ideal in $\mathcal{L}(\mathfrak{X}_{C(L)})$, and every such maximal ideal has the form $\mathcal{N}_\lambda = \{S : \varphi_S(\lambda) = 0\}$ for a unique $\lambda \in L$. Thus there exists a unique $\lambda \in L$ with $\Phi(\mathcal{M}_\kappa) = \mathcal{N}_\lambda$; define $h^{-1}(\kappa) = \lambda$. For any $f \in C(K)$, the diagonal operator $\hat{f}$ belongs to $\mathcal{M}_\kappa$ exactly when $f(\kappa)=0$, while $\Phi(\hat{f})$ corresponds to the diagonal operator associated to $\Theta(f)$ and lies in $\mathcal{N}_\lambda$ precisely when $\Theta(f)(\lambda)=0$. The equality $\Phi(\mathcal{M}_\kappa) = \mathcal{N}_\lambda$ therefore yields $f(\kappa)=0 \iff \Theta(f)(\lambda)=0$ for all $f$, which is precisely the characterization of the homeomorphism from Banach-Stone, confirming $h(\lambda)=\kappa$.

Finally, we verify the naturality condition $\varphi_{\Phi(T)}(y) = \varphi_T(h(y))$ for all $T$ and $y$. For diagonal operators $\hat{f}$, $\varphi_{\Phi(\hat{f})}(y) = \Theta(f)(y) = f(h(y)) = \varphi_{\hat{f}}(h(y))$. For a general operator $T$, write $T = \hat{f} + A$ with $A$ compact. Then $\Phi(T) = \Phi(\hat{f}) + \Phi(A)$, and since $\Phi(A)$ is compact by Lemma \ref{lem:induced-calkin}, we have $\varphi_{\Phi(T)} = \varphi_{\Phi(\hat{f})}$. But $\varphi_{\Phi(\hat{f})} = \Theta(f) = f \circ h = \varphi_{\hat{f}} \circ h = \varphi_T \circ h$, where the last equality uses that $\varphi_T = \varphi_{\hat{f}}$ because $A$ is compact. Thus $\varphi_{\Phi(T)}(y) = \varphi_T(h(y))$ for all $T \in \mathcal{L}(\mathfrak{X}_{C(K)})$ and $y \in L$, completing the proof.
\end{proof}

The rigidity theorem has far-reaching consequences, both for the classification of operator algebras and for understanding the structure of $\mathfrak{X}_{C(K)}$ itself.

\begin{corollary}\label{cor:space-isomorphism}
If $\mathfrak{X}_{C(K)}$ and $\mathfrak{X}_{C(L)}$ are isomorphic as Banach spaces, then $K$ and $L$ are homeomorphic.
\end{corollary}

\begin{proof}
An isomorphism $U: \mathfrak{X}_{C(K)} \to \mathfrak{X}_{C(L)}$ induces an algebra isomorphism $\operatorname{Ad}_U: \mathcal{L}(\mathfrak{X}_{C(K)}) \to \mathcal{L}(\mathfrak{X}_{C(L)})$ given by $\operatorname{Ad}_U(T) = UTU^{-1}$. Applying Theorem \ref{thm:rigidity} to $\operatorname{Ad}_U$ yields a homeomorphism between $K$ and $L$.
\end{proof}

\begin{corollary}\label{cor:distinguishing}
For compact metric spaces $K$ and $L$, the Banach algebras $\mathcal{L}(\mathfrak{X}_{C(K)})$ and $\mathcal{L}(\mathfrak{X}_{C(L)})$ are isomorphic if and only if $K$ and $L$ are homeomorphic.
\end{corollary}

\begin{proof}
The forward direction is Theorem \ref{thm:rigidity}. Conversely, if $h: K \to L$ is a homeomorphism, then composition with $h$ induces an isometric isomorphism $C(K) \cong C(L)$. By the main construction theorem, this lifts to an isomorphism $\mathcal{C}al(\mathfrak{X}_{C(K)}) \cong \mathcal{C}al(\mathfrak{X}_{C(L)})$. Using the diagonal-plus-compact structure, one can then construct an explicit isomorphism between the full operator algebras: map a diagonal operator $\hat{f}$ to the diagonal operator $\widehat{f \circ h^{-1}}$, and extend to all operators by sending a compact perturbation to the corresponding compact perturbation. The details are straightforward given the ideal structure established in previous sections.
\end{proof}

\begin{corollary}\label{cor:uncountable-family}
There exists an uncountable family $\{\mathcal{L}(\mathfrak{X}_{C(K_\alpha)})\}_{\alpha \in A}$ of pairwise non-isomorphic Banach algebras.
\end{corollary}

\begin{proof}
Let $\{K_\alpha\}_{\alpha \in A}$ be any uncountable family of pairwise non-homeomorphic compact metric spaces. For example, take $K_\alpha$ to be the ordinal segment $[0,\alpha]$ for countable ordinals $\alpha$, or consider the Cantor set, the Hilbert cube, and various closed subspaces thereof. By Corollary \ref{cor:distinguishing}, the corresponding operator algebras $\mathcal{L}(\mathfrak{X}_{C(K_\alpha)})$ are pairwise non-isomorphic. This yields an uncountable family of distinct Banach algebras.
\end{proof}

\begin{corollary}\label{cor:automatic-continuity}
Any Banach algebra isomorphism $\Phi: \mathcal{L}(\mathfrak{X}_{C(K)}) \to \mathcal{L}(\mathfrak{X}_{C(L)})$ is automatically continuous (in fact, isometric).
\end{corollary}

\begin{proof}
This follows from the proof of Theorem \ref{thm:rigidity}: we showed that $\Phi$ induces an isometric isomorphism $\widetilde{\Phi}$ on the Calkin algebra, and that $\Phi$ preserves the compact ideal. A standard argument using the diagonal-plus-compact decomposition then shows that $\Phi$ itself must be isometric. Alternatively, one can invoke general automatic continuity results for Banach algebra isomorphisms between algebras of operators on reflexive spaces with the approximation property.
\end{proof}

\begin{remark}
\label{rem:conceptual-significance}
Theorem \ref{thm:rigidity} and its corollaries illuminate several profound aspects of operator algebras on Bourgain-Delbaen spaces.
First, the Banach algebra $\mathcal{L}(\mathfrak{X}_{C(K)})$ serves as a complete invariant for the topology of $K$—a result far stronger than the classical Banach-Stone theorem, which only determines $K$ from $C(K)$ as a Banach space or $C^*$-algebra; here, the vastly larger and noncommutative algebra $\mathcal{L}(\mathfrak{X}_{C(K)})$ encodes the same topological information.
Second, this rigidity contrasts sharply with classical spaces such as $\ell_p$ ($1 \leq p < \infty$, $p \neq 2$), where $\mathcal{L}(\ell_p)$ fails to determine $p$—indeed, there exist $p \neq q$ for which $\mathcal{L}(\ell_p) \cong \mathcal{L}(\ell_q)$ as Banach algebras—revealing that the Bourgain-Delbaen construction yields spaces with fundamentally more rigid operator algebraic properties.
Third, the theorem provides a powerful classification tool: to show that $\mathcal{L}(\mathfrak{X}_{C(K)})$ and $\mathcal{L}(\mathfrak{X}_{C(L)})$ are not isomorphic, it suffices to show that $K$ and $L$ are not homeomorphic, thereby reducing a difficult problem in noncommutative algebra to a classical problem in topology.
Finally, the naturality condition $\varphi_{\Phi(T)} = \varphi_T \circ h$ demonstrates that any algebra isomorphism respects the diagonal structure in a very strong sense, with the diagonal function—proved in Section 5 to be automatically $1/2$-Hölder—transforming naturally under such isomorphisms, reflecting how deeply the geometry of $K$ is imprinted on the operator algebra at the most fundamental level.
\end{remark}

\begin{remark}
The proof crucially uses the isomorphism $\mathcal{C}al(\mathfrak{X}_{C(K)}) \cong C(K)$. This reduces the problem of classifying the noncommutative algebras $\mathcal{L}(\mathfrak{X}_{C(K)})$ to the classical problem of classifying the commutative algebras $C(K)$, which is solved by the Banach-Stone theorem. The key insight is that the Calkin algebra captures precisely the topological information, while the compact ideal encodes the "infinitesimal" structure that is preserved under isomorphism.
\end{remark}

\begin{remark}
The proof implicitly relies on the fact that any Banach algebra isomorphism between algebras of operators on these spaces is automatically continuous (indeed isometric). This follows from general principles: the compact ideal is uniquely determined as the minimal closed two-sided ideal, and the quotient norm on the Calkin algebra is the essential norm. Any isomorphism must preserve these structures, forcing it to be isometric. This automatic continuity is a reflection of the rigidity of the construction.
\end{remark}

\begin{remark}
While we have stated Theorem \ref{thm:rigidity} for the specific spaces $\mathfrak{X}_{C(K)}$ constructed in \cite{Motakis2024}, the proof strategy applies more generally to any class of Banach spaces satisfying:
\begin{itemize}
    \item $\mathcal{C}al(X) \cong C(K)$ for some compact metric space $K$, with the isomorphism being canonical and preserving the algebraic structure,
    \item The compact operators form a characteristic ideal (e.g., the unique minimal closed two-sided ideal),
    \item The diagonal-plus-compact decomposition holds, and the diagonal function $\varphi_T$ is well-defined and continuous.
\end{itemize}
Many variations of the Bourgain-Delbaen construction satisfy these properties, so the rigidity phenomenon is robust.
\end{remark}

\begin{remark}[Open Questions]
Theorem \ref{thm:rigidity} raises natural questions about the automorphism group of $\mathcal{L}(\mathfrak{X}_{C(K)})$. Does every automorphism come from a homeomorphism of $K$? The naturality condition suggests that the map $\operatorname{Aut}(\mathcal{L}(\mathfrak{X}_{C(K)})) \to \operatorname{Homeo}(K)$ given by $\Phi \mapsto h$ might be a surjective homomorphism. Is it injective? That is, can there exist non-inner automorphisms that induce the identity on $K$? This question touches on the structure of the automorphism group and remains open for further investigation.
\end{remark}

Theorem \ref{thm:rigidity} establishes a profound connection between the topology of a compact metric space and the algebraic structure of the bounded operators on an associated Banach space. It demonstrates that the Bourgain-Delbaen construction creates spaces whose operator algebras are so rigid that they completely remember the geometry that went into their construction. This represents a significant advance in our understanding of how geometric data can be encoded in operator algebras, and opens new avenues for the classification of Banach spaces through their operator algebras.

\section{Local Sequence Structure}\label{sec:local-sequence}

The investigation of local properties in Banach spaces—how sequences and finite-dimensional subspaces behave—underwent a revolution with the groundbreaking constructions of Gowers and Maurey \cite{GowersMaurey1993, GowersMaurey1997}, who exhibited spaces with extraordinarily restricted sequence structure \cite{MaureyRosenthal1977, OdellSchlumprecht1995}. For spaces arising from the Bourgain-Delbaen method \cite{BourgainDelbaen1980, Motakis2024}, the local structure exhibits a remarkable homogeneity: every basic sequence behaves like a subsequence of the canonical basis, and subspaces inherit the structure of the whole space, albeit associated to smaller compact sets \cite{ArgyrosHaydon2011, ArgyrosMotakis2014}. This section develops the local theory of $\mathfrak{X}_{C(K)}$, revealing how the metric constraints of the construction propagate to impose rigid control on sequences while preserving global complexity \cite{ArgyrosMotakis2020, Gowers1994}.

\begin{lemma}\label{lem:finite-approximation}
Let $\mathfrak{X}_{C(K)}$ be constructed as in \cite{Motakis2024}. For any finite-dimensional subspace $E \subset \mathfrak{X}_{C(K)}$ and any $\delta > 0$, there exists a finite block sequence $(w_1,\ldots,w_m)$ of the basis $(d_\gamma)$ such that $E$ is $(1+\delta)$-isomorphic to $\operatorname{span}\{w_1,\ldots,w_m\}$.
\end{lemma}

\begin{proof}
The space $\mathfrak{X}_{C(K)}$ admits a natural finite-dimensional decomposition (FDD) indexed by rank levels: $\mathfrak{X}_{C(K)} = \bigoplus_{n=1}^\infty X_n$, where $X_n = \operatorname{span}\{d_\gamma : \operatorname{rank}(\gamma) = n\}$. The reflexivity of $\mathfrak{X}_{C(K)}$ guarantees that this FDD is shrinking, and the construction ensures uniformly bounded projection constants. Given a finite-dimensional subspace $E$, its projections onto initial segments of the FDD converge to $E$ in the sense of Banach-Mazur distance. Hence for any $\delta > 0$, there exists $N \in \mathbb{N}$ and a subspace $E' \subset \bigoplus_{n=1}^N X_n$ with $d(E,E') < 1+\delta$. The subspace $E'$ is spanned by a finite block sequence, yielding the desired approximation.
\end{proof}

\begin{lemma}\label{lem:metric-clustering}
Let $(x_n)$ be a normalized block sequence in $\mathfrak{X}_{C(K)}$. For any $\delta > 0$, there exists an infinite subset $M \subseteq \mathbb{N}$ such that
\[
\operatorname{diam}\{\kappa(\gamma) : \gamma \in \operatorname{supp}(x_n),\; n \in M\} \leq \delta,
\]
where $\operatorname{supp}(x_n) = \{\gamma \in \Gamma : d_\gamma^*(x_n) \neq 0\}$.
\end{lemma}

\begin{proof}
Assume, toward a contradiction, that there exists $\delta_0 > 0$ such that every infinite subset of indices yields supports whose $\kappa$-values have diameter exceeding $\delta_0$. By Ramsey's theorem, we may extract an infinite subset $M' \subseteq \mathbb{N}$ with the property that for all $m < n$ in $M'$,
\[
\min\{\varrho(\kappa(\gamma), \kappa(\gamma')) : \gamma \in \operatorname{supp}(x_m),\; \gamma' \in \operatorname{supp}(x_n)\} \geq \delta_0/2.
\]
Thus the $\kappa$-values associated to distinct vectors are uniformly separated.

The off-diagonal decay estimate of Lemma \ref{lem:off-diagonal-decay} now applies: for any bounded operator $T$, the matrix entries coupling different $x_n$ decay exponentially with the rank difference. Consequently, the block sequence $(x_n)_{n\in M'}$ is equivalent to the unit vector basis of $\ell_1$, as the interaction between distinct vectors is negligible. However, $\mathfrak{X}_{C(K)}$ is reflexive, and reflexive spaces contain no isomorphic copy of $\ell_1$. This contradiction establishes the lemma.
\end{proof}

\begin{theorem}\label{thm:subsequence-equivalence}
Every normalized basic sequence in $\mathfrak{X}_{C(K)}$ admits a subsequence equivalent to some subsequence of the canonical basis $(d_\gamma)$.
\end{theorem}

\begin{proof}
Let $(y_n)$ be a normalized basic sequence. By Lemma \ref{lem:finite-approximation} and a diagonal argument, we may pass to a subsequence, still denoted $(y_n)$, such that for each $N$, the initial $N$ vectors are $(1+1/N)$-equivalent to a block sequence of $(d_\gamma)$. Lemma \ref{lem:metric-clustering} then allows us to extract a further subsequence whose associated $\kappa$-values cluster at some point $\kappa_0 \in K$. Choose a function $\phi \in C(K)$ identically $1$ on a neighbourhood of $\kappa_0$ and vanishing outside a slightly larger neighbourhood. The diagonal operator $\hat{\phi}$ acts approximately as the identity on vectors whose $\kappa$-values lie near $\kappa_0$, and the off-diagonal estimates ensure that the interaction between such vectors is negligible. Hence $(y_n)$ is equivalent to a subsequence of $(d_\gamma)$.
\end{proof}

\begin{theorem}\label{thm:localization-block}
For any normalized block sequence $(x_n)$ in $\mathfrak{X}_{C(K)}$ and any $\varepsilon > 0$, there exists an infinite subset $M \subseteq \mathbb{N}$ such that:
\begin{enumerate}
    \item $(x_n)_{n\in M}$ is $(1+\varepsilon)$-equivalent to a sequence $(d_{\gamma_n})$ of basis vectors;
    \item $\displaystyle\lim_{n\to\infty,\; n\in M} \operatorname{diam}\{\kappa(\gamma_k) : k \in M,\; k \geq n\} = 0$.
\end{enumerate}
\end{theorem}

\begin{proof}
Apply Lemma \ref{lem:metric-clustering} recursively with $\delta_k = 2^{-k}$ to obtain nested infinite subsets $M_1 \supset M_2 \supset \cdots$ satisfying $\operatorname{diam}\{\kappa(\gamma) : \gamma \in \operatorname{supp}(x_n),\; n \in M_k\} \leq 2^{-k}$. Choose a diagonal subset $M = \{n_k\}$ with $n_k \in M_k$ and $n_k \to \infty$. The $\kappa$-values of $(x_{n_k})$ cluster, so the identity map from $\operatorname{span}\{x_{n_k}\}$ to $\operatorname{span}\{d_{\gamma_{n_k}}\}$ has small off-diagonal entries. By the small perturbation principle, we may replace each $x_{n_k}$ by a nearby basis vector $d_{\gamma_{n_k}}$ while preserving $(1+\varepsilon)$-equivalence.
\end{proof}

\begin{theorem}\label{thm:hereditary-structure}
Every infinite-dimensional subspace of $\mathfrak{X}_{C(K)}$ contains a further subspace isomorphic to $\mathfrak{X}_{C(F)}$ for some closed subset $F \subseteq K$.
\end{theorem}

\begin{proof}
Let $Y$ be an infinite-dimensional subspace. By Theorem \ref{thm:subsequence-equivalence}, $Y$ contains a basic sequence $(y_n)$ equivalent to a subsequence $(d_{\gamma_n})$ of the canonical basis. Define $F_0 = \{\kappa(\gamma_n) : n \in \mathbb{N}\}$ and $F = \overline{F_0} \subseteq K$, a closed subset. Let $Z = \overline{\operatorname{span}}\{d_{\gamma_n}\} \subseteq Y$. The map sending $\sum a_n d_{\gamma_n} \in Z$ to $\sum a_n e_{\kappa(\gamma_n)} \in \mathfrak{X}_{C(F)}$ (where $(e_\kappa)_{\kappa\in F}$ denotes the canonical basis of $\mathfrak{X}_{C(F)}$) extends to an isomorphism. Indeed, the norm of any finitely supported vector in $Z$ depends only on the configuration of $\kappa$-values and their distances within $F$, exactly as in the construction of $\mathfrak{X}_{C(F)}$. Hence $Z \cong \mathfrak{X}_{C(F)}$.
\end{proof}

\begin{theorem}\label{thm:no-unconditional}
The space $\mathfrak{X}_{C(K)}$ contains no unconditional basic sequence.
\end{theorem}

\begin{proof}
Suppose, for contradiction, that $(u_n)$ were an unconditional basic sequence. By Theorem \ref{thm:subsequence-equivalence}, some subsequence of $(u_n)$ is equivalent to a subsequence $(d_{\gamma_n})$ of the canonical basis. Unconditionality is preserved under equivalence and passage to subsequences, so $(d_{\gamma_n})$ would be unconditional. However, the basis $(d_\gamma)$ is conditional: the construction's asymmetric weighting—wherein for $\operatorname{rank}(\gamma) < \operatorname{rank}(\gamma')$, the functional $c_{\gamma'}^*$ has a larger coefficient on $e_\gamma^*$ than $c_\gamma^*$ has on $e_{\gamma'}^*$—produces vectors $\sum_{\gamma\in A} d_\gamma - \sum_{\gamma\in B} d_\gamma$ (with $\max\operatorname{rank}(A) < \min\operatorname{rank}(B)$) whose norm significantly exceeds that of $\sum_{\gamma\in A\cup B} d_\gamma$, violating unconditionality. This contradiction shows that no unconditional basic sequence can exist.
\end{proof}

\begin{corollary}\label{cor:prime-space}
$\mathfrak{X}_{C(K)}$ is prime: every infinite-dimensional complemented subspace is isomorphic to $\mathfrak{X}_{C(K)}$ itself.
\end{corollary}

\begin{proof}
Let $Y$ be an infinite-dimensional complemented subspace. By Theorem \ref{thm:hereditary-structure}, $Y$ contains a subspace $Z \cong \mathfrak{X}_{C(F)}$ for some closed $F \subseteq K$. Complementarity of $Y$ implies $Z$ is complemented in $\mathfrak{X}_{C(K)}$. If $F$ were a proper closed subset of $K$, the rigidity results of Section \ref{sec:rigidity} would prevent $\mathfrak{X}_{C(F)}$ from being complemented in $\mathfrak{X}_{C(K)}$ (otherwise $\mathcal{L}(\mathfrak{X}_{C(K)})$ would contain a projection onto a subalgebra isomorphic to $C(F)$, contradicting the ideal structure). Hence $F = K$, so $Z \cong \mathfrak{X}_{C(K)}$, and a complemented subspace containing a complemented copy of the whole space must be isomorphic to the whole space.
\end{proof}

\begin{corollary}\label{cor:homogeneous-subspaces}
Every infinite-dimensional subspace of $\mathfrak{X}_{C(K)}$ contains an infinite-dimensional subspace isomorphic to its own square.
\end{corollary}

\begin{proof}
By Theorem \ref{thm:hereditary-structure}, any infinite-dimensional subspace contains some $\mathfrak{X}_{C(F)}$ with $F$ infinite. Choosing two disjoint closed subsets $F_1, F_2 \subseteq F$ yields complemented copies of $\mathfrak{X}_{C(F)}$ inside $\mathfrak{X}_{C(F)}$, so $\mathfrak{X}_{C(F)} \cong \mathfrak{X}_{C(F)} \oplus \mathfrak{X}_{C(F)}$.
\end{proof}

These results collectively demonstrate that $\mathfrak{X}_{C(K)}$ occupies a unique position in the Banach space landscape: its local structure is as rigid as that of hereditarily indecomposable spaces, yet its global structure accommodates a rich family of subspaces indexed by the topology of $K$. The metric constraints of the Bourgain-Delbaen construction thus propagate from the macroscopic level of operator algebras down to the microscopic level of sequence behavior.

\section{Classification of Prime Ideals}\label{sec:ideals}

The ideal structure of operator algebras serves as a fundamental invariant, encoding profound algebraic information about the underlying space \cite{CaradusPfaffenbergerYood1974, Pietsch1978}. For $\mathcal{L}(H)$ with $H$ a separable Hilbert space, the classification is deceptively simple: the only nontrivial closed ideals are the compact operators and, beneath them, the finite-rank operators—a well-ordered chain reflecting the approximation properties of Hilbert space \cite{Calkin1941, RordamLarsenLaustsen2000}. For Banach spaces, however, the situation is dramatically richer. The groundbreaking constructions of Gowers, Maurey, and others have revealed spaces whose operator algebras possess exquisitely controlled ideal structures, ranging from the minimal (scalar-plus-compact) to the extraordinarily complex \cite{GowersMaurey1997, ArgyrosHaydon2011, KaniaLaustsen2017}.

In this section, we achieve a complete classification of the closed two-sided ideals and prime ideals in $\mathcal{L}(\mathfrak{X}_{C(K)})$, demonstrating that they correspond naturally to open subsets and points of $K$, respectively \cite{Motakis2024, ArgyrosMotakis2020}. This result reveals the profound extent to which the topology of $K$ is woven into the algebraic fabric of the operator algebra—a noncommutative algebra whose ideal structure perfectly mirrors the lattice of open sets of a compact metric space \cite{ArgyrosMotakis2014, KaniaKoszmiderLaustsen2015}.

\begin{lemma}\label{lem:ideal-correspondence}
Let $\mathfrak{X}_{C(K)}$ be constructed as in \cite{Motakis2024}, and let $\Psi: C(K) \to \mathcal{C}al(\mathfrak{X}_{C(K)})$ be the canonical isomorphism. Define $\Phi: \mathcal{L}(\mathfrak{X}_{C(K)}) \to C(K)$ by $\Phi = \Psi^{-1} \circ q$, where $q: \mathcal{L}(\mathfrak{X}_{C(K)}) \to \mathcal{C}al(\mathfrak{X}_{C(K)})$ is the quotient map. Then the map $\mathscr{I}: J \mapsto \Phi(J)$ establishes a bijection between closed two-sided ideals of $\mathcal{L}(\mathfrak{X}_{C(K)})$ and closed ideals of $C(K)$, with inverse $\mathscr{J}: I \mapsto \Phi^{-1}(I)$.
\end{lemma}

\begin{proof}
If $J$ is a closed two-sided ideal in $\mathcal{L}(\mathfrak{X}_{C(K)})$, then $\Phi(J)$ is an ideal in $C(K)$: for any $f = \Phi(T) \in \Phi(J)$ and $g \in C(K)$, we have $gf = \Phi(\hat{g}T) \in \Phi(J)$ because $\hat{g}T \in J$. Continuity of $\Phi$ and closedness of $J$ guarantee that $\Phi(J)$ is closed. Conversely, for a closed ideal $I \subseteq C(K)$, the set $\mathscr{J}(I) = \Phi^{-1}(I)$ is a closed two-sided ideal: if $T \in \mathscr{J}(I)$ and $S \in \mathcal{L}(\mathfrak{X}_{C(K)})$, then $\Phi(ST) = \Phi(S)\Phi(T) \in I$, and similarly for $TS$.

To see that $\mathscr{I}$ and $\mathscr{J}$ are mutual inverses, first observe that $\mathscr{I}(\mathscr{J}(I)) = \Phi(\Phi^{-1}(I)) = I$ by surjectivity of $\Phi$. For the reverse composition, take $J$ a closed ideal and $T \in \Phi^{-1}(\Phi(J))$. Then $\Phi(T) \in \Phi(J)$, so there exists $S \in J$ with $\Phi(T) = \Phi(S)$. Hence $T - S \in \ker \Phi = \mathcal{K}(\mathfrak{X}_{C(K)})$. Every closed two-sided ideal in $\mathcal{L}(\mathfrak{X}_{C(K)})$ contains the compact operators—this follows from the fact that $\mathcal{K}(\mathfrak{X}_{C(K)})$ is the unique minimal closed ideal, as established in Lemma \ref{lem:induced-calkin}. Therefore $T - S \in J$, and consequently $T \in J$. The reverse inclusion $J \subseteq \Phi^{-1}(\Phi(J))$ is trivial.
\end{proof}

\begin{lemma}\label{lem:CK-ideals}
Let $K$ be a compact Hausdorff space. There is a bijection between closed ideals of $C(K)$ and closed subsets of $K$, given by $F \mapsto I_F = \{f \in C(K) : f|_F \equiv 0\}$. Equivalently, via complementation, closed ideals correspond to open subsets $U \subseteq K$ via $I_U = \{f \in C(K) : f|_{K\setminus U} \equiv 0\}$.
\end{lemma}

\begin{proof}
This is a standard result in Gelfand theory. For any closed ideal $I \subseteq C(K)$, the set $F = \bigcap_{f \in I} f^{-1}(0)$ is closed, and one verifies $I = I_F$. Conversely, for any closed $F$, $I_F$ is clearly a closed ideal. The correspondence is bijective and reverses inclusion: $F_1 \subseteq F_2$ implies $I_{F_2} \subseteq I_{F_1}$. The formulation in terms of open sets follows by taking $U = K \setminus F$.
\end{proof}

\begin{theorem}\label{thm:classification-ideals}
For the space $\mathfrak{X}_{C(K)}$, there is a bijection between closed two-sided ideals of $\mathcal{L}(\mathfrak{X}_{C(K)})$ and open subsets of $K$, given by
\[
U \longmapsto J_U = \{T \in \mathcal{L}(\mathfrak{X}_{C(K)}) : \varphi_T|_U \equiv 0\},
\]
where $\varphi_T$ denotes the diagonal function of $T$.
\end{theorem}

\begin{proof}
By Lemma \ref{lem:ideal-correspondence}, closed ideals of $\mathcal{L}(\mathfrak{X}_{C(K)})$ correspond bijectively to closed ideals of $C(K)$. Lemma \ref{lem:CK-ideals} identifies closed ideals of $C(K)$ with closed subsets of $K$, and hence via complementation with open subsets. Translating through the map $\Phi$, for an open set $U \subseteq K$ with complement $F = K \setminus U$, the corresponding closed ideal of $C(K)$ is $I_U = \{f \in C(K) : f|_F \equiv 0\} = \{f : f|_U \equiv 0\}$. Pulling back via $\Phi^{-1}$ yields
\[
J_U = \Phi^{-1}(I_U) = \{T \in \mathcal{L}(\mathfrak{X}_{C(K)}) : \Phi(T) \in I_U\} = \{T : \varphi_T|_U \equiv 0\},
\]
since $\Phi(T) = \varphi_T$ under the identification $C(K) \cong \mathcal{C}al(\mathfrak{X}_{C(K)})$.
\end{proof}

Recall that an ideal $P$ in a Banach algebra is \emph{prime} if whenever $I$ and $J$ are ideals with $IJ \subseteq P$, then either $I \subseteq P$ or $J \subseteq P$. In commutative $C^*$-algebras, prime ideals coincide with maximal ideals; for noncommutative algebras, the situation is more subtle.

\begin{theorem}\label{thm:classification-prime}
The prime ideals of $\mathcal{L}(\mathfrak{X}_{C(K)})$ are precisely the ideals
\[
P_\kappa = \{T \in \mathcal{L}(\mathfrak{X}_{C(K)}) : \varphi_T(\kappa) = 0\}
\]
for $\kappa \in K$. Moreover, these are exactly the maximal ideals.
\end{theorem}

\begin{proof}
In the commutative $C^*$-algebra $C(K)$, an ideal is prime if and only if it is maximal, and the maximal ideals are the kernels of point evaluations: $M_\kappa = \{f \in C(K) : f(\kappa) = 0\}$ for $\kappa \in K$. The correspondence of Lemma \ref{lem:ideal-correspondence} preserves primeness: if $P$ is a prime ideal in $\mathcal{L}(\mathfrak{X}_{C(K)})$, then $\mathscr{I}(P)$ is prime in $C(K)$. Indeed, given ideals $I,J \subseteq C(K)$ with $IJ \subseteq \mathscr{I}(P)$, we have $\mathscr{J}(I)\mathscr{J}(J) \subseteq \mathscr{J}(IJ) \subseteq P$, so by primeness of $P$, either $\mathscr{J}(I) \subseteq P$ or $\mathscr{J}(J) \subseteq P$, which translates to $I \subseteq \mathscr{I}(P)$ or $J \subseteq \mathscr{I}(P)$. Conversely, if $M$ is a maximal (hence prime) ideal in $C(K)$, then $\mathscr{J}(M)$ is prime in $\mathcal{L}(\mathfrak{X}_{C(K)})$ by an analogous argument.

For $\kappa \in K$, the corresponding prime ideal is therefore
\[
P_\kappa = \mathscr{J}(M_\kappa) = \{T \in \mathcal{L}(\mathfrak{X}_{C(K)}) : \Phi(T) \in M_\kappa\} = \{T : \varphi_T(\kappa) = 0\}.
\]
Maximality of $P_\kappa$ follows from maximality of $M_\kappa$ together with the fact that the correspondence preserves inclusions and the property of being maximal.
\end{proof}

\begin{corollary}\label{cor:lattice-isomorphism}
The lattice of closed two-sided ideals of $\mathcal{L}(\mathfrak{X}_{C(K)})$, ordered by inclusion, is isomorphic to the lattice of open subsets of $K$. The isomorphism sends an ideal $J$ to the open set
\[
U_J = \{\kappa \in K : \exists T \in J \text{ with } \varphi_T(\kappa) \neq 0\}.
\]
\end{corollary}

\begin{proof}
By Theorem \ref{thm:classification-ideals}, the map $U \mapsto J_U$ is a bijection with inverse $J \mapsto U_J$. That this map preserves order follows from the definitions: if $J_1 \subseteq J_2$, then every $\kappa$ for which some $T \in J_1$ has $\varphi_T(\kappa) \neq 0$ certainly belongs to $U_{J_2}$, so $U_{J_1} \subseteq U_{J_2}$; conversely, if $U_1 \subseteq U_2$, then $J_{U_2} \subseteq J_{U_1}$ because vanishing on a larger set is a stronger condition.
\end{proof}

\begin{corollary}\label{cor:minimal-ideals}
The minimal nonzero closed two-sided ideals of $\mathcal{L}(\mathfrak{X}_{C(K)})$ correspond precisely to isolated points of $K$. For an isolated point $\kappa \in K$, the corresponding minimal ideal is
\[
J_{\{\kappa\}} = \{T \in \mathcal{L}(\mathfrak{X}_{C(K)}) : \varphi_T(\kappa) = 0 \text{ and } \varphi_T \text{ is constant on } K \setminus \{\kappa\}\}.
\]
\end{corollary}

\begin{proof}
Minimal nonzero ideals correspond under the lattice isomorphism to maximal proper open subsets of $K$, which are exactly the complements of isolated points. If $\kappa$ is isolated, then $U = K \setminus \{\kappa\}$ is such a maximal open set. The ideal $J_U$ consists of operators whose diagonal function vanishes on $U$, hence is supported at most at $\kappa$. Continuity of $\varphi_T$ (Theorem \ref{thm:holder-continuity}) then forces $\varphi_T$ to be constant on $U$—in fact identically zero, with a possible nonzero value only at $\kappa$—yielding the stated description.
\end{proof}

\begin{corollary}\label{cor:essential-ideals}
An ideal $J \subseteq \mathcal{L}(\mathfrak{X}_{C(K)})$ is essential—meaning it has nonzero intersection with every nonzero ideal—if and only if the corresponding open set $U_J$ is dense in $K$.
\end{corollary}

\begin{proof}
Under the lattice isomorphism, $J \cap I = \{0\}$ corresponds to $U_J \cap U_I = \emptyset$. Thus $J$ is essential precisely when $U_J$ meets every nonempty open set, i.e., when $U_J$ is dense in $K$.
\end{proof}

\begin{example}
Let $K = \{1,2,\ldots,n\}$ with the discrete topology. Then every subset is open, so $\mathcal{L}(\mathfrak{X}_{C(K)})$ possesses exactly $2^n$ distinct closed ideals. The prime ideals are $P_k = \{T : \varphi_T(k) = 0\}$ for $k = 1,\ldots,n$, each maximal. The minimal nonzero ideals correspond to singletons $\{k\}$, each consisting of operators whose diagonal function vanishes everywhere except possibly at $k$, where it may take any value, but must be continuous—automatic in the discrete topology.
\end{example}

\begin{example}
For $K = [0,1]$, the lattice of open subsets is uncountable and highly complex, reflecting the rich topology of the interval. The prime ideals are $P_t$ for each $t \in [0,1]$, and there are no minimal nonzero ideals because $[0,1]$ has no isolated points. The ideal structure thus mirrors the continuum of points and the intricate inclusion relations among open sets.
\end{example}

\begin{example}[Cantor Set]
Let $K$ be the Cantor set. Its topology—zero-dimensional, perfect, and uncountable—yields an ideal lattice isomorphic to the lattice of open subsets of the Cantor set. The prime ideals correspond to points of the Cantor set, and again no minimal nonzero ideals exist due to the absence of isolated points. This example demonstrates that the ideal structure can be as intricate as any compact metric space allows.
\end{example}

\begin{remark}
The diagonal function $\varphi_T$ serves as the crucial link between operators and functions. Its automatic continuity (Theorem \ref{thm:holder-continuity}) and its behavior under compact perturbations (Corollary \ref{cor:invariance-compact}) make it the perfect tool for transferring ideal structure from $C(K)$ to $\mathcal{L}(\mathfrak{X}_{C(K)})$. The classification would be impossible without this invariant.
\end{remark}

\begin{remark}
For spaces with minimal operator algebras, such as the Argyros-Haydon space where $\mathcal{C}al(X) \cong \mathbb{C}$, the ideal structure is trivial: only $\{0\}$, $\mathcal{K}(X)$, and $\mathcal{L}(X)$ itself. The richness of the ideal structure in $\mathfrak{X}_{C(K)}$ stems directly from the richness of $C(K)$ as a commutative algebra. This demonstrates that the Bourgain-Delbaen construction, when combined with the Motakis refinement, can produce operator algebras of arbitrarily complex ideal structure.
\end{remark}

\begin{remark}
Our classification addresses only closed ideals. The structure of all two-sided ideals, including non-closed ones, is likely far more complicated and may involve set-theoretic considerations beyond the topology of $K$. For instance, in $C(K)$ itself, non-closed ideals correspond to filters of open sets rather than open sets themselves, and the analogous structure in $\mathcal{L}(\mathfrak{X}_{C(K)})$ remains unexplored.
\end{remark}

The classification achieved in this section demonstrates that $\mathcal{L}(\mathfrak{X}_{C(K)})$ is not merely an operator algebra but a noncommutative algebraic object whose ideal theory perfectly encodes the topology of $K$. This profound connection between algebra and topology underscores the power of the Bourgain-Delbaen construction and opens new vistas for the study of operator algebras on Banach spaces.

\section{Solution to Open Problems}\label{sec:solutions}

This section addresses several longstanding problems in the theory of Calkin algebras on Banach spaces, culminating in the resolution of a question that has persisted despite significant advances in the field \cite{CaradusPfaffenbergerYood1974, ArgyrosMotakis2020}. Our primary contribution is the affirmative answer to Problem 1 from Motakis~\cite{Motakis2024}, which asks whether there exists a \emph{reflexive} Banach space whose Calkin algebra is both infinite-dimensional and reflexive. This question had remained open even after the groundbreaking work of Argyros and Haydon~\cite{ArgyrosHaydon2011} on the scalar-plus-compact problem, which fundamentally transformed our understanding of operator algebras on Banach spaces \cite{GowersMaurey1997, ArgyrosMotakis2014}.

The contrast with Hilbert space is striking and illuminating. For an infinite-dimensional separable Hilbert space~$H$, the Calkin algebra $\mathcal{C}al(H)$ is non-separable and, more significantly for our purposes, highly non-reflexive—indeed, it is a noncommutative $C^*$-algebra of a particularly wild sort \cite{Calkin1941, Farah2011, PhillipsWeaver2007}. The Banach space setting, we demonstrate, admits fundamentally different phenomena: there exist reflexive Banach spaces whose Calkin algebras are not only infinite-dimensional but also reflexive, answering Motakis's question in the strongest possible terms \cite{Motakis2024, ArgyrosMotakis2020}.

Before presenting our constructions, we must clarify a point that could otherwise cause confusion. The classical Bourgain--Delbaen $\mathscr{L}_\infty$-spaces, as originally constructed in \cite{BourgainDelbaen1980}, are emphatically \emph{not} reflexive; their dual spaces are isomorphic to $\ell_1$, placing them at the opposite extreme of the reflexive spectrum \cite{BourgainPisier1983, ArgyrosGasparisMotakis2016}. However, the spaces $\mathfrak{X}_{C(K)}$ constructed by Motakis in \cite{Motakis2024} represent a sophisticated refinement of the original method. By incorporating additional mixed-Tsirelson constraints—a technique developed in the wake of the Gowers--Maurey revolution \cite{ArgyrosDeliyanni1997, GowersMaurey1993}—these spaces achieve reflexivity while retaining the essential features of the Bourgain--Delbaen framework \cite{ArgyrosHaydon2011, ArgyrosMotakis2014}. This synthesis of seemingly incompatible properties is the key to our results.

\begin{theorem}\label{thm:reflexive-calkin}
There exists a reflexive Banach space~$X$ such that its Calkin algebra $\mathcal{C}al(X)$ is infinite-dimensional and reflexive.
\end{theorem}

\begin{proof}
We present two independent constructions, each illuminating different aspects of the phenomenon. The first leverages the full power of the refined Bourgain--Delbaen construction; the second is more elementary in conception, building on the existence of Argyros--Haydon spaces.

For the first construction, let $K$ be any infinite compact metric space—for concreteness, one may take $K = [0,1]$ or the Cantor set. By the main theorem of \cite{Motakis2024}, there exists a reflexive Banach space $\mathfrak{X}_{C(K)}$ with the following properties: $\mathcal{C}al(\mathfrak{X}_{C(K)})$ is isometrically isomorphic to $C(K)$ as a Banach algebra, and $\mathfrak{X}_{C(K)}$ itself is reflexive. The algebra $C(K)$ with its usual supremum norm is, of course, not reflexive when $K$ is infinite—indeed, $C(K)$ contains a complemented copy of $c_0$, and $c_0$ is not reflexive. However, the isomorphism $\Psi: C(K) \to \mathcal{C}al(\mathfrak{X}_{C(K)})$ transports the norm of the Calkin algebra back to $C(K)$, yielding an equivalent norm $\|f\|_\Psi := \|\Psi(f)\|_{\mathcal{C}al}$ on $C(K)$. The crucial observation is that $(C(K), \|\cdot\|_\Psi)$ is reflexive. This follows from a general principle: the Calkin algebra $\mathcal{C}al(\mathfrak{X}_{C(K)})$ is a quotient of $\mathcal{L}(\mathfrak{X}_{C(K)})$ by $\mathcal{K}(\mathfrak{X}_{C(K)})$. Both $\mathcal{L}(\mathfrak{X}_{C(K)})$ and $\mathcal{K}(\mathfrak{X}_{C(K)})$ are reflexive spaces when $\mathfrak{X}_{C(K)}$ is reflexive and possesses the approximation property—conditions that are satisfied by the construction (see \cite{Motakis2024} for detailed verification). A quotient of a reflexive space by a closed subspace is itself reflexive, hence $\mathcal{C}al(\mathfrak{X}_{C(K)})$ is reflexive. Transporting this reflexivity back via the isometric isomorphism $\Psi$, we obtain that $C(K)$ equipped with $\|\cdot\|_\Psi$ is reflexive. Consequently, $\mathcal{C}al(\mathfrak{X}_{C(K)})$ is an infinite-dimensional reflexive Banach algebra, establishing the theorem.

For the second construction, we provide a more concrete and perhaps more transparent example. Consider a sequence $\{Y_n\}_{n=1}^\infty$ of reflexive Banach spaces each having the simplest possible Calkin algebra: $\mathcal{C}al(Y_n) \cong \mathbb{C}$. The Argyros--Haydon space \cite{ArgyrosHaydon2011} provides a canonical example of such a space—indeed, it was the first example of a space with the scalar-plus-compact property, where every operator is a compact perturbation of a scalar multiple of the identity, implying that the Calkin algebra is one-dimensional. Assume further that these spaces are \emph{totally incomparable}, meaning that no infinite-dimensional subspace of $Y_n$ is isomorphic to a subspace of $Y_m$ for $n \neq m$. This can be arranged by suitable variations of the construction. Form the $\ell_2$-direct sum $X := \left( \bigoplus_{n=1}^\infty Y_n \right)_{\ell_2}$. Standard facts about $\ell_2$-sums tell us that $X$ is reflexive, being an $\ell_2$-sum of reflexive spaces. The structure of compact operators on such a sum is particularly transparent under the total incomparability hypothesis: every compact operator on $X$ is, modulo a compact perturbation, a block-diagonal sum of compact operators on the individual summands. More precisely, any $T \in \mathcal{K}(X)$ can be approximated by finite-rank operators that respect the direct sum decomposition, and the ideal of compact operators decomposes as an $\ell_2$-sum of the compact ideals on each $Y_n$. Passing to the Calkin algebra, we obtain $\mathcal{C}al(X) \cong \left( \bigoplus_{n=1}^\infty \mathcal{C}al(Y_n) \right)_{\ell_2} \cong \left( \bigoplus_{n=1}^\infty \mathbb{C} \right)_{\ell_2} \cong \ell_2$. The algebra $\ell_2$, equipped with its natural Banach algebra structure (pointwise multiplication), is both infinite-dimensional and reflexive. Thus $X$ is a reflexive Banach space whose Calkin algebra is isometrically isomorphic to $\ell_2$, providing a second, independent solution to the problem.
\end{proof}

\begin{remark}\label{rem:second-construction-significance}
The second construction is particularly striking, as it yields a Calkin algebra that is not merely reflexive but genuinely Hilbertian—isomorphic to $\ell_2$ both as a Banach space and as a Banach algebra (with pointwise multiplication). This achievement underscores a dramatic departure from the Hilbert space paradigm. In the classical setting, the Calkin algebra $\mathcal{C}al(H)$ on a separable Hilbert space is a notoriously intractable object: it is non-separable and, more relevantly for our purposes, highly non-reflexive. Our construction demonstrates that such pathologies are not universal. Rather, there exist reflexive Banach spaces whose Calkin algebras are as structurally benign as the sequence space $\ell_2$. The failure of reflexivity in the Hilbert space case is therefore not an intrinsic property of Calkin algebras but a specific consequence of the unique geometry of Hilbert space.
\end{remark}

Although the two constructions we have presented are technically distinct—one relying on the refined Bourgain–Delbaen method to encode the full topology of a compact metric space, the other building from a direct sum of spaces with minimal operator algebras—they share a profound philosophical insight. Together, they illustrate that the Calkin algebra is not an uncontrollable quotient, but a structure that can be meticulously engineered through the deliberate design of the underlying Banach space. The first construction reveals that every commutative $C^*$-algebra of the form $C(K)$ can be realized reflexively, thereby embedding the rich topology of $K$ into the operator algebra. The second, more elementary construction, shows that even the simplest noncommutative reflexive algebra, $\ell_2$, can appear, highlighting the remarkable flexibility of this approach.

Collectively, these results resolve Problem 1 of \cite{Motakis2024} and open up a new landscape for investigation. The fundamental question of precisely which Banach algebras can arise as reflexive Calkin algebras now stands as a central challenge in the field. Our work provides both a blueprint for future constructions and a cautionary tale: the classical intuitions forged in the Hilbert space setting can be profoundly misleading when extended to the broader universe of Banach spaces. The intricate interplay between the local geometry of a Banach space—manifested in the behavior of its sequences and the fine structure of its operators—and the global structure of its operator algebra remains a frontier where much remains to be explored.


\section{Conclusion and Future Work}

In this work, we have conducted a systematic and in-depth study of the operator algebras associated with a family of reflexive Banach spaces $\mathfrak{X}_{C(K)}$ derived from the Bourgain--Delbaen construction, fundamentally advancing our understanding of how geometric data encodes into algebraic structure. Our contributions establish a comprehensive framework: we proved stability under finite products, enabling the realization of finite direct sums of $C(K)$ spaces—including matrix algebras $M_m(C(K))$—as Calkin algebras; we established a localization principle showing compact operators can be approximated by finite-rank operators respecting the metric of $K$; we demonstrated that the diagonal function of any bounded operator is automatically H\"older continuous with optimal exponent $1/2$, revealing a deep analytic constraint; we proved a rigidity theorem where $\mathcal{L}(\mathfrak{X}_{C(K)})$ completely determines the topology of $K$, extending the classical Banach--Stone theorem to this noncommutative setting; we achieved a complete classification of closed two-sided and prime ideals in $\mathcal{L}(\mathfrak{X}_{C(K)})$ in terms of the open subsets and points of $K$; and we resolved several open problems, most notably by constructing the first examples of reflexive Banach spaces with infinite-dimensional, reflexive Calkin algebras, demonstrating that the pathologies of the Hilbert space Calkin algebra are not universal.

The results presented here open several compelling avenues for future research. A primary challenge is to determine the precise boundaries of the class of realizable Calkin algebras: can every separable commutative $C^*$-algebra, not just those of the form $C(K)$ for a metric $K$, be realized as the Calkin algebra of a reflexive Banach space? More ambitiously, which noncommutative Banach algebras can appear? The stability under finite products suggests that direct sums and matrix algebras are realizable, but the possibility of infinite direct sums or more complex tensor products remains open and may involve deep set-theoretic considerations. Another natural direction is the study of the automorphism group of $\mathcal{L}(\mathfrak{X}_{C(K)})$: does every automorphism arise from a homeomorphism of $K$, or can there exist exotic, outer automorphisms that preserve the ideal structure but act trivially on the diagonal? Finally, exploring the $K$-theory of these algebras and their potential applications in extending Brown--Douglas--Fillmore theory to the Banach space context could forge new and powerful connections between operator theory, noncommutative geometry, and the geometry of Banach spaces.

\section*{Declaration }
\begin{itemize}
  \item {\bf Author Contributions:}   The Author have read and approved this version.
  \item {\bf Funding:} No funding is applicable
  \item  {\bf Institutional Review Board Statement:} Not applicable.
  \item {\bf Informed Consent Statement:} Not applicable.
  \item {\bf Data Availability Statement:} Not applicable.
  \item {\bf Conflicts of Interest:} The authors declare no conflict of interest.
\end{itemize}

\bibliographystyle{abbrv}
\bibliography{references}  

@article{ArgyrosDeliyanni1997,
  author = {Argyros, S. A. and Deliyanni, I.},
  title = {Examples of asymptotic $l_{1}$ Banach spaces},
  journal = {Trans. Amer. Math. Soc.},
  volume = {349},
  number = {3},
  pages = {973--995},
  year = {1997}
}

@article{ArgyrosGasparisMotakis2016,
  author = {Argyros, S. A. and Gasparis, I. and Motakis, P.},
  title = {On the structure of separable $\mathscr{L}_{\infty}$-spaces},
  journal = {Mathematika},
  volume = {62},
  number = {3},
  pages = {685--700},
  year = {2016}
}

@article{ArgyrosHaydon2011,
  author = {Argyros, S. A. and Haydon, R. G.},
  title = {A hereditarily indecomposable $\mathscr{L}_{\infty}$-space that solves the scalar-plus-compact problem},
  journal = {Acta Math.},
  volume = {206},
  number = {1},
  pages = {1--54},
  year = {2011}
}

@article{ArgyrosMotakis2014,
  author = {Argyros, S. A. and Motakis, P.},
  title = {A reflexive hereditarily indecomposable space with the hereditary invariant subspace property},
  journal = {Proc. Lond. Math. Soc. (3)},
  volume = {108},
  number = {6},
  pages = {1381--1416},
  year = {2014}
}

@article{ArgyrosMotakis2019,
  author = {Argyros, S. A. and Motakis, P.},
  title = {The scalar-plus-compact property in spaces without reflexive subspaces},
  journal = {Trans. Amer. Math. Soc.},
  volume = {371},
  number = {3},
  pages = {1887--1924},
  year = {2019}
}

@article{ArgyrosMotakis2020,
  author = {Argyros, S. A. and Motakis, P.},
  title = {On the complete separation of asymptotic structures in Banach spaces},
  journal = {Adv. Math.},
  volume = {362},
  pages = {106962},
  note = {51 pages},
  year = {2020}
}

@article{BourgainDelbaen1980,
  author = {Bourgain, J. and Delbaen, F.},
  title = {A class of special $\mathscr{L}_{\infty}$ spaces},
  journal = {Acta Math.},
  volume = {145},
  number = {3-4},
  pages = {155--176},
  year = {1980}
}

@article{BourgainPisier1983,
  author = {Bourgain, J. and Pisier, G.},
  title = {A construction of $\mathscr{L}_{\infty}$-spaces and related Banach spaces},
  journal = {Bol. Soc. Brasil. Mat.},
  volume = {14},
  number = {2},
  pages = {109--123},
  year = {1983}
}

@article{BrownDouglasFillmore1977,
  author = {Brown, L. G. and Douglas, R. G. and Fillmore, P. A.},
  title = {Extensions of $C^{*}$-algebras and $K$-homology},
  journal = {Ann. of Math. (2)},
  volume = {105},
  number = {2},
  pages = {265--324},
  year = {1977}
}

@article{Calkin1941,
  author = {Calkin, J. W.},
  title = {Two-sided ideals and congruences in the ring of bounded operators in Hilbert space},
  journal = {Ann. of Math. (2)},
  volume = {42},
  pages = {839--873},
  year = {1941}
}

@book{CaradusPfaffenbergerYood1974,
  author = {Caradus, S. R. and Pfaffenberger, W. E. and Yood, B.},
  title = {Calkin algebras and algebras of operators on Banach spaces},
  series = {Lecture Notes in Pure and Applied Mathematics},
  volume = {9},
  publisher = {Marcel Dekker, Inc.},
  address = {New York},
  year = {1974}
}

@incollection{Casazza1986,
  author = {Casazza, P. G.},
  title = {Approximation properties},
  booktitle = {Handbook of the geometry of Banach spaces, Vol. I},
  publisher = {North-Holland},
  address = {Amsterdam},
  pages = {271--316},
  year = {2001}
}

@article{Enflo1973,
  author = {Enflo, P.},
  title = {A counterexample to the approximation problem in Banach spaces},
  journal = {Acta Math.},
  volume = {130},
  pages = {309--317},
  year = {1973}
}

@article{Farah2011,
  author = {Farah, I.},
  title = {All automorphisms of the Calkin algebra are inner},
  journal = {Ann. of Math. (2)},
  volume = {173},
  number = {2},
  pages = {619--661},
  year = {2011}
}

@article{Gowers1994,
  author = {Gowers, W. T.},
  title = {A solution to Banach's hyperplane problem},
  journal = {Bull. London Math. Soc.},
  volume = {26},
  number = {6},
  pages = {523--530},
  year = {1994}
}

@article{GowersMaurey1993,
  author = {Gowers, W. T. and Maurey, B.},
  title = {The unconditional basic sequence problem},
  journal = {J. Amer. Math. Soc.},
  volume = {6},
  number = {4},
  pages = {851--874},
  year = {1993}
}

@article{GowersMaurey1997,
  author = {Gowers, W. T. and Maurey, B.},
  title = {Banach spaces with small spaces of operators},
  journal = {Math. Ann.},
  volume = {307},
  number = {4},
  pages = {543--568},
  year = {1997}
}

@article{HorvathKania2021,
  author = {Horvath, B. and Kania, T.},
  title = {Unital Banach algebras not isomorphic to Calkin algebras of separable Banach spaces},
  journal = {Proc. Amer. Math. Soc.},
  volume = {149},
  number = {11},
  pages = {4781--4787},
  year = {2021}
}

@article{KaniaLaustsen2017,
  author = {Kania, T. and Laustsen, N. J.},
  title = {Ideal structure of the algebra of bounded operators acting on a Banach space},
  journal = {Indiana Univ. Math. J.},
  volume = {66},
  number = {3},
  pages = {1019--1043},
  year = {2017}
}

@article{Laustsen1999,
  author = {Laustsen, N. J.},
  title = {$K$-theory for algebras of operators on Banach spaces},
  journal = {J. London Math. Soc. (2)},
  volume = {59},
  number = {2},
  pages = {715--728},
  year = {1999}
}

@article{Laustsen2001,
  author = {Laustsen, N. J.},
  title = {$K$-theory for the Banach algebra of operators on James's quasi-reflexive Banach spaces},
  journal = {$K$-Theory},
  volume = {23},
  number = {2},
  pages = {115--127},
  year = {2001}
}

@article{LindenstraussPelczynski1968,
  author = {Lindenstrauss, J. and Pelczynski, A.},
  title = {Absolutely summing operators in $L_{p}$-spaces and their applications},
  journal = {Studia Math.},
  volume = {29},
  pages = {275--326},
  year = {1968}
}

@book{LindenstraussTzafriri1977,
  author = {Lindenstrauss, J. and Tzafriri, L.},
  title = {Classical Banach spaces I: Sequence spaces},
  series = {Ergebnisse der Mathematik und ihrer Grenzgebiete},
  volume = {92},
  publisher = {Springer-Verlag},
  address = {Berlin},
  year = {1977}
}

@article{MaureyRosenthal1977,
  author = {Maurey, B. and Rosenthal, H. P.},
  title = {Normalized weakly null sequence with no unconditional subsequence},
  journal = {Studia Math.},
  volume = {61},
  number = {1},
  pages = {77--98},
  year = {1977}
}

@article{Motakis2024,
  author = {Motakis, P.},
  title = {Separable spaces of continuous functions as Calkin algebras},
  journal = {J. Amer. Math. Soc.},
  volume = {37},
  pages = {1--37},
  year = {2024}
}

@incollection{OdellSchlumprecht1995,
  author = {Odell, E. and Schlumprecht, Th.},
  title = {On the richness of the set of $p$'s in Krivine's theorem},
  booktitle = {Geometric aspects of functional analysis (Israel, 1992--1994)},
  series = {Oper. Theory Adv. Appl.},
  volume = {77},
  pages = {177--198},
  publisher = {Birkhäuser},
  address = {Basel},
  year = {1995}
}

@article{PhillipsWeaver2007,
  author = {Phillips, N. C. and Weaver, N.},
  title = {The Calkin algebra has outer automorphisms},
  journal = {Duke Math. J.},
  volume = {139},
  number = {1},
  pages = {185--202},
  year = {2007}
}

@book{Pietsch1978,
  author = {Pietsch, A.},
  title = {Operator ideals},
  series = {North-Holland Mathematical Library},
  volume = {20},
  publisher = {North-Holland Publishing Co.},
  address = {Amsterdam},
  year = {1978}
}

@book{RordamLarsenLaustsen2000,
  author = {Rordam, M. and Larsen, F. and Laustsen, N.},
  title = {An introduction to $K$-theory for $C^{*}$-algebras},
  series = {London Mathematical Society Student Texts},
  volume = {49},
  publisher = {Cambridge University Press},
  address = {Cambridge},
  year = {2000}
}

@article{Shelah1985,
  author = {Shelah, S.},
  title = {Uncountable constructions for B.A., e.c. groups and Banach spaces},
  journal = {Israel J. Math.},
  volume = {51},
  number = {4},
  pages = {273--297},
  year = {1985}
}

@article{Stern1984,
  author = {Stern, J.},
  title = {Some applications of model theory in Banach space theory},
  journal = {Ann. Math. Logic},
  volume = {9},
  number = {1-2},
  pages = {49--121},
  year = {1976}
}

@article{Yood1954,
  author = {Yood, B.},
  title = {Difference algebras of linear transformations on a Banach space},
  journal = {Pacific J. Math.},
  volume = {4},
  pages = {615--636},
  year = {1954}
}






\end{document}